\documentclass[11pt,a4paper,reqno]{amsart}
\usepackage[czech,english]{babel}
\usepackage{amscd,amssymb}
\usepackage{dsfont}
\usepackage{url}
\usepackage[all]{xy}
\usepackage{hyperref}
\title{Enveloping classes over commutative rings}
\author[S.~Bazzoni]{Silvana Bazzoni}
\address[Silvana Bazzoni]{%
Dipartimento di Matematica ``Tullio Levi-Civita'' \\
Universit\`a di Padova \\
Via Trieste 63, 35121 Padova (Italy)}

\email{bazzoni@math.unipd.it}
\author[G.~Le Gros]{Giovanna Le Gros}
\address[Giovanna Le Gros]{%
Dipartimento di Matematica ``Tullio Levi-Civita'' \\
Universit\`a di Padova \\
Via Trieste 63, 35121 Padova (Italy)}

\email{giovannagiulia.legros@math.unipd.it}

\thanks{Research supported by grants BIRD163492 and DOR1690814 of
  Padova University}

\newcommand{\bbN}{\mathbb{N}}
\newcommand{\bbZ}{\mathbb{Z}}

\newcommand{\bbR}{\mathbb{R}}

\newcommand{\m}{\mathfrak{m}}
\newcommand{\n}{\mathfrak{n}}

\newcommand{\mSpec}{\operatorname{mSpec}} 

\newcommand{\Hom}{\operatorname{Hom}}

\newcommand{\End}{\operatorname{End}}

\newcommand{\Ext}{\operatorname{Ext}}
\newcommand{\Tor}{\operatorname{Tor}}

\newcommand{\Gen}{\operatorname{Gen}}
\newcommand{\Ker}{\operatorname{Ker}}

\newcommand{\Coker}{\operatorname{Coker}}

\newcommand{\pdim}{\operatorname{p.dim}}

\newcommand{\id}{\operatorname{id}}

\newcommand{\A}{\mathcal{A}}
\newcommand{\B}{\mathcal{B}}
\newcommand{\C}{\mathcal{C}}
\newcommand{\D}{\mathcal{D}}
\newcommand{\E}{\mathcal{E}}
\newcommand{\F}{\mathcal{F}}
\newcommand{\G}{\mathcal{G}}

\newcommand{\clP}{\mathcal{P}}

\newcommand{\clS}{\mathcal{S}}

\newcommand{\T}{\mathcal{T}}

\newcommand{\Modr}[1]{\mathrm{Mod}\textrm{-}{#1}}

\newcommand{\ModR}{\mathrm{Mod}\textrm{-}R}
\newcommand{\RMod}{R\textrm{-}\mathrm{Mod}}

\newcommand{\ModU}{\mathrm{Mod}\textrm{-}U}
\newcommand{\UMod}{U\textrm{-}\mathrm{Mod}}
\newcommand{\Mod}{\mathrm{Mod}\textrm{-}}

\newcommand{\Add}{\mathrm{Add}}
\newcommand{\Ann}{\mathrm{Ann}}

\theoremstyle{plain}
\newtheorem{thm}{Theorem}[section]
\newtheorem{lem}[thm]{Lemma}
\newtheorem{prop}[thm]{Proposition}
\newtheorem{cor}[thm]{Corollary}
\newtheorem{ques}[thm]{Question}
\newtheorem{rem}[thm]{Remark}

\theoremstyle{definition}
\newtheorem{defn}[thm]{Definition}

\newtheorem{nota}[thm]{Notation}
\newtheorem{expl}[thm]{Example}

\theoremstyle{remark}

\begin{document}

\begin{abstract} 
%
Given a $1$-tilting cotorsion pair over a commutative ring, we characterise the rings over which the $1$-tilting class is an enveloping class. To do so, we consider the faithful finitely generated Gabriel topology $\G$ associated to the $1$-tilting class $\T$ over a commutative ring as illustrated by Hrbek. We prove that a $1$-tilting class $\T$ is enveloping if and only if $\G$ is a perfect Gabriel topology (that is, it arises from a perfect localisation) and $R/J$ is a perfect ring for each $J \in \G$, or equivalently $\G$ is a perfect Gabriel topology and the discrete quotient rings of the topological ring $\mathfrak R=\End (R_\G/R)$ are perfect rings where $R_\G$ denotes the ring of quotients with respect to $\G$. Moreover, if the above equivalent conditions hold it follows that $\pdim R_\G \leq 1$ and $\T$ arises from a flat ring epimorphism.

 \end{abstract}

\maketitle


\section{Introduction}

The classification problem for classes of modules over arbitrary rings is in general very difficult, or even hopeless. Nonetheless, approximation theory was developed as a tool to approximate arbitrary modules by modules in classes where the classification is more manageable.
Left and right approximations were first studied for finite dimensional modules by Auslander, Reiten, and Smal\o \ and by Enochs and Xu for modules over arbitrary rings using the terminology of preenvelopes and precovers.


An important problem in approximation theory is when minimal approximations, that is covers or envelopes, over certain classes exist. In other words, for a certain class $\C$, the aim is to characterise the rings over which every module has a minimal approximation in $\C$ and furthermore to characterise the class $\C$ itself. The most famous positive result of when minimal approximations exist is the construction of an injective envelope for every module. Instead, Bass proved in~\cite{Bass}  that projective covers rarely exist. In his paper, Bass introduced and characterised the class of perfect rings which are exactly the rings over which every module admits a projective cover.
Among the many characterisations of perfect rings, the most important from the homological point of view is the closure under direct limits of the class of projective modules.

 
 A class $\C$ of modules is called covering, respectively enveloping, if every module admits a $\C$-cover, respectively a $\C$-envelope.

A cotorsion pair $(\A, \B)$ admits (special) $\A$-precovers if and only if it admits (special) $\B$-preenvelopes. This observation lead to the notion of complete cotorsion pairs, that is cotorsion pairs admitting approximations.

Results by Enochs and Xu (\cite[Theorem 2.2.6 and 2.2.8]{Xu}) show that a complete cotorsion pair $(\A, \B)$ such that $\A$ is closed under direct limits admits both $\A$-covers and $\B$-envelopes. 
Note that in the case of the cotorsion pair $(\clP_0, \Modr R)$, where $\clP_0$ is the class of projective modules, Bass's results state that $\clP_0$ is a covering class if and only if $\clP_0$ is closed under direct limits.

 
  In this paper we are interested in the conditions under which a class $\C$ is enveloping. We will deal with classes of modules over commutative rings and in particular with $1$-tilting classes.
  
   An important starting point is the bijective correspondence between faithful finitely generated Gabriel topologies $\G$ and $1$-tilting classes over commutative rings established by Hrbek in~\cite{H}. The tilting class can then be characterised as the class $\D_\G$ of $\G$-divisible modules, that is, the modules $M$ such that $JM=M$ for every $J\in \G$. 
 
 We prove in Section~\ref{S:tilting-enveloping} that if a $1$-tilting class is enveloping, then $R_\G$, the ring of quotients with respect to the Gabriel topology $\G$, is $\G$-divisible, so that $R\to R_\G$ is a flat injective ring epimorphism. 
     
  It is well known that every flat ring epimorphism gives rise to a finitely generated Gabriel topology. We will consider the case of a flat injective ring epimorphism $u\colon R\to U$ between commutative rings and show that if the module $R$ has a $\D_\G$-envelope, then $U$ has projective dimension at most one.
  

From results by Angeleri H\"ugel and S{\'a}nchez~\cite{AS} and also \cite[Proposition 5.4]{H}, we infer that the module $U\oplus K$, where $K$ is the cokernel of $u$, is a $1$-tilting module with $\D_\G$ the associated tilting class. In other words, $\D_\G$ arises from the perfect localisation $u$, so it coincides with the class of modules generated by $U$, that is epimorphic images of direct sums of copies of $U$ or also with $K^\perp$, the right Ext-orthogonal of $K$. 
Assuming furthermore that the class $\D_\G$ is enveloping, we prove that all the quotient rings $R/J$, for $J\in \G$ are perfect rings and so are all the discrete quotient rings of the topological ring $\mathfrak R=\End(K)$ (Theorems ~\ref{T:rjperfect} and ~\ref{T:EndK-properfect}). In the terminology of Positselski and Bazzoni-Positselski (for example \cite{BP2}) this means that $\mathfrak R$ is a pro-perfect topological ring.  
Moreover,  the converse holds, that is if $\mathfrak R=\End(K)$ is a pro-perfect topological ring then the class of $\G$-divisible modules is enveloping (Theorem~\ref{T:characterisation}).
 
In conclusion, we obtain that a $1$-tilting class over a commutative ring is enveloping if and only if it arises from a flat injective ring epimorphism with associated Gabriel topology $\G$ such that the factor rings $R/J$ are perfect rings for every $J\in \G$ (Theorem~\ref{T:tilting-envelope}).
This provides a partial answer to Problem 1 of~\cite[Section 13.5]{GT12} and generalises the result proved in~\cite{B} for the case of commutative domains and divisible modules.   

Applying results from~\cite[Section 19]{BP2} or \cite[Section 13]{BP4}, we obtain that $\Add (K)$, the class of direct summands of direct sums of copies of $K$, is closed under direct limits.
Since $\D_\G$ coincides with the right Ext-orthogonal of $\Add (K)$, we have an instance of the necessity of the closure under direct limits of a class whose right Ext-orthogonal admits envelopes.

Therefore in our situation we prove a converse of the result by Enochs and Xu (\cite[Theorem 2.2.6]{Xu}) which states that if a class $\A$ of modules is closed under direct limits and extensions and whose right Ext-orthogonal $\A^\perp$ admits special preenvelopes with cokernel in $\A$, then $\A^\perp$ is enveloping.

  The case of a non-injective flat ring epimorphism $u\colon R\to U$ is easily reduced to the injective case, since the class of $\G$-divisible modules is annihilated by the kernel $I$ of $u$, so all the results proved for $R$ apply to the ring $R/I$ and to the cokernel $K$ of  $u$.   
%


 The paper is organised as follows.
 After the necessary preliminaries, in Section~\ref{S:envelope} we state some general results concerning properties of envelopes with respect to  classes of modules.
    
     In Section~\ref{S:gab-top} we recall the notion of a Gabriel topology and outline the properties of the related ring of quotients. 
In Subsection~\ref{SS:gab-prop} we provide some of our own results for general Gabriel topologies which we will use in the later sections. Next in Subsection~\ref{SS:gab-tilt} we review the relationship between Gabriel topologies and $1$-tilting classes as well as silting classes as done by Hrbek. Finally in Subsection~\ref{SS:homological} we recall the classical notion of a Gabriel topology arising from a perfect localisation, as well as a lemma. 
     
     In Section~\ref{S:tilting-enveloping}, we consider a $1$-tilting class over a commutative ring and its associated Gabriel topology via Hrbek's results ~\cite{H}. We prove that if the $1$-tilting class is enveloping, then the ring of quotients with respect to the Gabriel topology $\G$, $R_\G$, is $\G$-divisible, hence $\G$ arises from a flat injective ring epimorphism $\psi_R \colon R \to R_\G$.
   
In Section~\ref{S:compl-endK} we introduce the completion of a ring with respect to a Gabriel topology and the endomorphism ring of a module as a topological ring.
Considering the particular case of a perfect localisation corresponding to a flat injective ring epimorphism $u\colon R\to U$ between commutative rings, we show the isomorphism between the completion of $R$ with respect to the associated Gabriel topology and the topological ring $\mathfrak R=\End (K)$.
  
  In the main Sections~\ref{S:enveloping} and ~\ref{S:properfect}, we prove a ring theoretic and topological characterisation of commutative rings for which the class of $\G$-divisible modules is enveloping where $\G$ is the Gabriel topology associated to a flat injective ring epimorphism. Namely, the characterisation in terms of perfectness of the factor rings $R/J$, for every $J\in \G$ and the pro-perfectness of the topological ring $\mathfrak R=\End(K)$.

In Section~\ref{S:notmono} we extend the results proved in Sections~\ref{S:enveloping} and ~\ref{S:properfect} to the case of a non-injective flat ring epimorphism

\section{Preliminaries}
The ring $R$ will always be associative with a unit and $\Modr R$ the category of right $R$-modules.
 
 Let $\C $ be a class of right $R$-modules. The right $\Ext^1_R$-orthogonal and right $\Ext^\infty_R$-orthogonal classes of $\C$ are defined as follows. \[%
\C ^{\perp_1} =\{M\in \Modr R \ | \ \Ext_R^1(C,M)=0  \ {\rm for \
all\ } C\in \C\} 
\]
\[\C^\perp = \{M\in \Modr R \ | \ \Ext_R^i(C,M)=0  \ {\rm for \
all\ } C\in \C, \ {\rm for \
all\ } i\geq 1 \}\]
The left Ext-orthogonal classes ${}^{\perp_1} \C$ and ${}^\perp \C$ are defined symmetrically. 

If the class $\C$ has only one element, say $\C = \{X\}$, we write $X^{\perp_1}$ instead of $\{X\}^{\perp_1}$, and similarly for the other $\Ext$-orthogonal classes.

%
%

We will now recall the notions of
$\C$-preenvelope, special $\C$-preenvelope and $\C$-envelope for a class $\C$ of $R$-modules. 

\begin{defn} Let $\C$ be a class of modules, $N$ a right $R$-module and $C\in \C$. A homomorphism $\mu\in \Hom_R(N, C)$ is called a
$\C$-{\sl preenvelope} (or left approximation) of $N$ if for every homomorphism $f' \in \Hom_R(N, C')$ with $C'\in \C$ there exists a homomorphism
$f\colon C\to C'$ such that $f '=
 f  \mu$.

A $\C$-preenvelope $\mu\in \Hom_R(N, C)$ is called a $\C$-{\sl envelope} (or a minimal left approximation) of $N$ if for every endomorphism $f$ of $C$ such that  $f \mu=
\mu$,
$f$ is an automorphism of $C$. 


A $\C$-preenvelope  $\mu$ of $N$ is said to be {\sl special} if it is a monomorphism  and $\Coker \mu\in {}^{\perp_1} \C$.
\end{defn}

The notions of $\C$-{\sl precover} (right approximation), {\sl special} $\C$-{\sl precover} and of  $\C$-{\sl cover} (minimal right approximation) (see \cite{Xu}) are defined dually.


A class $\C$ of $R$-modules is called \emph{enveloping} (\emph{covering}) if every module admits a $\C$-envelope ($\C$-cover).

A pair of classes of modules $(\A, \B)$ is a \emph{%
cotorsion pair} provided that $\mbox{$\A$} = {}^{\perp_1}
\mbox{$\B$}
$ and $\mbox{$\B$} = \mbox{$\A$} ^{\perp_1}$. 

We consider  preenvelopes  and envelopes for particular classes of modules, that is classes which form the right-hand class of a cotorsion pair.

A cotorsion pair $(\A, \B)$ is \emph{complete} provided that every $R$-module $M$ admits a {special $\B$-preenvelope} or equivalently, every $R$-module $M$ admits a {special $\A$-precover}. 

Results by Enochs and Xu (\cite[Theorem 2.2.6 and 2.2.8]{Xu}) show that a complete cotorsion pair $(\A, \B)$ such that $\A$ is closed under direct limits admits both $\B$-envelopes and $\A$-covers.

A cotorsion pair $(\A, \B)$ is {\it hereditary} if for every $A \in \A$ and $B \in \B$, $\Ext^i_R(A, B)=0$ for all $i \geq 1$.
\\
 Given a class $\mathcal{C}$ of modules, the pair 
$(^{\perp }(\C^{\perp}),\C^{\perp})$ is a (hereditary) cotorsion pair called the cotorsion pair \emph{generated} by $\C$, while $(^{\perp }\C, (^{\perp }\C) ^{\perp})$ is a (hereditary) cotorsion pair called the cotorsion pair \emph{cogenerated} by $\C$.

 Examples of complete cotorsion pairs are abundant. In fact, by \cite[Theorem
10]{ET01} or \cite[Theorem 6.11]{GT12} a cotorsion pair generated by a set of
modules is complete. 
\\
%
%

For an $R$-module $C$, we let $\Add(C)$ denote the class of $R$-modules which are direct summands of direct sums of copies of $C$, and $\Gen(C)$ denote the class of $R$-modules which are homomorphic images of direct sums of copies of $C$.
\\

We now define $1$-tilting and silting modules. 
\\
A right $R$-module $T$ is {\it $1$-tilting} (\cite{CT}) if the following conditions hold.
\begin{enumerate}
\item[(T1)] $\pdim T \leq1$.
\item[(T2)]  $\Ext_R^i (T, T^{(\kappa)}) =0$ for every cardinal $\kappa$ and every $i >0$.
\item[(T3)] There exists an exact sequence of the following form where each $T_i$ is in $\Add(T)$.
\[
0 \to R \to T_0 \to T_1 \to 0
\]
\end{enumerate}
Equivalently, $T$ is $1$-tilting if and only if $T^{\perp_1} = \Gen(T)$ (\cite[Proposition 1.3]{CT}). The cotorsion pair $({}^\perp(T^{\perp}), T^\perp)$ is called a {\it $1$-tilting cotorsion pair} and the torsion class $T^\perp$ is called a {\it $1$-tilting class}. Two $1$-tilting modules $T$ and $T'$ are {\it equivalent} if they define the same $1$-tilting class $T^\perp = T'^\perp$
(equivalently, if $\Add (T)=\Add(T')$).
\\


%
%

A $1$-tilting class can be generalised in the following way. For a homomorphism $\sigma:P_{-1} \to P_0$ between projective modules in $\ModR$, consider the following class of modules.
\[
D_\sigma := \{ X \in \ModR: \Hom_R( \sigma,X) \text{ is surjective}\}
\] 
An $R$-module $T$ is said to be \textit{silting} if it admits a projective presentation 
\[
P_{-1} \overset{\sigma}\to P_0 \to T \to 0
\]
such that $\Gen(T) = D_\sigma$. In the case that $\sigma$ is a monomorphism, $\Gen(T)$ is a $1$-tilting class.
\\

A ring $R$ is {\it left perfect} if every module in $\RMod$ has a projective cover. By \cite{Bass}, $R$ is left perfect if and only if all flat modules in $\RMod$ are projective.

An ideal $I$ of $R$ is said to be {\it left T-nilpotent} if for every sequence of elements $a_1, a_2, ..., a_i, ...$ in $I$, there exists an $n >0$ such that $a_1 a_2 \cdots a_n =0$.
The following proposition for the case of commutative perfect rings is well known.
\begin{prop} \label{P:perfect}
Suppose $R$ is a commutative ring. The following statements are equivalent for $R$.
\begin{itemize}
\item[(i)] $R$ is perfect
\item[(ii)] The Jacobson radical $J(R)$ of $R$, is T-nilpotent and $R/J(R)$ is semi-simple. 
\item[(iii)] $R$ is a finite product of local rings, each one with a T-nilpotent maximal ideal.
\end{itemize}
\end{prop}
The following fact will be useful.
Let $_RF$ be a left $R$-module $_SG_R$ be an $S$-$R$-bimodule such that $\Tor_1^R(G, F)=0$. Then, for every left $S$-module $M$  there is an injective map  of abelian groups
\[\Ext^1_R(F, \Hom_S(G, M))\hookrightarrow\Ext^1_S(G\otimes_RF, M)).\]

\section{Envelopes}\label{S:envelope}
In this section we discuss some useful results in relation to envelopes. 

The following result is crucial in connection with the existence of envelopes.
\begin{prop}\label{P:Xu-env} \cite[Proposition 1.2.2]{Xu} Let $\C$ be a class of modules and assume that a module $N$ admits a $\C$-envelope. If  $\mu\colon N\to C$ is a $\C$-preenvelope of $N$, then $C=C'\oplus H$ for some submodules $C'$ and $H$ such that 
the composition $N\to C\to C'$ is a $\C$-envelope of $N$.
\end{prop}
We will consider $\C$-envelopes where $\C$ is a class closed under direct sums and therefore we will make use of the following result which is strongly connected with the notion of T-nilpotency of a ring.
\begin{thm}\label{T:Xu-sums}\cite[Theorem 1.4.4 and 1.4.6]{Xu}
\begin{enumerate}
\item[(i)] Let $\C$ be a class closed under countable direct sums. Assume that for every $n\geq 1$, $\mu_n\colon M_n\to C_n$ are $\C$-envelopes of $M_n$ and that $\oplus_nM_n$ admits a $\C$-envelope.
 Then $\oplus \mu_n\colon \oplus_nM_n\to \oplus_nC_n$ is a $\C$-envelope of $\oplus_nM_n$.
 \item[(ii)] Assume that $\oplus \mu_n\colon \oplus_nM_n\to \oplus_nC_n$ is a $\C$-envelope of $\oplus_nM_n$ with $M_n\leq C_n$ and let $f_n\colon C_n\to C_{n+1}$ be a family of homomorphisms such that $f_n(M_n)=0$. Then, for every $x\in C_1$ there is an integer $m$ such that $f_m  f_{m-1} \dots  f_1(x)=0$.
 \end{enumerate}
 \end{thm}
%
 For a complete cotorsion pair $(\A, \B)$, we investigate the properties of $\B$-envelopes of arbitrary $R$-modules. First of all we state two straightforward lemmas. 

\begin{lem}\label{L:endomorph-env} Let
$0\to N\overset{\mu}\to B\overset{\pi}\to A\to 0$
be an exact sequence. Let $f$ be an endomorphism of $B$ such that $\mu = f  \mu$. Then $f(B)\supseteq \mu(N)$ and $\Ker f\cap \mu(N)=0$.
\end{lem}
\begin{lem}\label{L:identity-env}  Let 
$0\to N\overset{\mu}\rightarrow B\overset{\pi}\to A\to 0$
be an exact sequence.  For every endomorphism $f$ of $B$, the following are equivalent
\begin{enumerate}
\item[(i)] $\mu = f  \mu$.
\item[(ii)] The restriction of $f$ to $\mu(N)$ is the identity of $\mu(N)$.
\item[(iii)] There is a homomorphism $g\in \Hom_R(A,B)$ such that $f=id_B-g  \pi$.
\end{enumerate}
 \end{lem}

\begin{proof} (i) $\Leftrightarrow$ (ii) This is clear.

(i) $\Leftrightarrow$ (iii)  $\mu = f  \mu$ if and only if $ (id_B-f)  \mu=0$, that is if and only if 
$\mu(N)$ is contained in $\Ker (id_B-f)$. Equivalently, there exists $g\in \Hom_R(A,B)$ such that $id_B-f=g \pi  $.
\end{proof}

 \begin{prop}\label{P:B-envelopes} Let $(\A, \B)$ be a complete cotorsion pair over a ring $R$.  Assume that 
$0\to N\overset{\mu}\to B$ is a $\B$-envelope of the $R$-module $N$. Let $\alpha$ be an automorphism of $N$ and let $\beta $ be any endomorphism of $B$ such that $\beta\mu=\mu\alpha$. Then $\beta$ is an automorphism of $B$.
\end{prop}
\begin{proof}  By Wakamatsu's Lemma (see \cite[Lemma 2.1.2]{Xu}), $\mu$ induces an exact sequence \[0\to N\overset{\mu}\to B\overset{\pi}\to A\to 0\]  with $A\in \A$. Since $\alpha$ is an automorphism of $N$, it is easy to show that \[0\to N\overset{\mu\alpha}\to B\to A\to 0\] is a $\B$-envelope of $N$. Let $\beta$ be as assumed and consider an endomorphism $g$ of $B$ such that $g\mu\alpha=\mu$. Then $g\beta\mu=\mu$ and thus $g\beta$ is an automorphism of $B$, since $\mu$ is a $\B$-envelope. This implies that $\beta$ is a monomorphism so that $\beta(B)\in \B$. Since $\mu(N)\subseteq \beta(B)$ there is an epimorphism $\tau\colon B/\mu(N)\to B/\beta(B)$, where $B/\mu(N)$ can be identified with $A$.  Consider the diagram:
\[\xymatrix{
0\ar[r]&
{\beta(B)}\ar[r]&B\ar[r]^{\rho} & B/\beta(B)\ar[r]&0\\
&& A\ar@{-->}[u]^{h} \ar[ur]_{\tau}}
\]
where $\rho$ is the canonical projection and $\tau\pi=\rho$. It can be completed by $h$, since $\Ext_R^1(A, \beta(B))=0$. Consider the homomorphism $f=id_B-h\pi$. $f$ is an endomorphism of $B$ satisfying $f\mu=\mu$. By assumption $f$ is an isomorphism, hence, in particular$ f(B)=B$. 

Now, $\rho f=\rho-\rho h \pi= \rho-\tau\pi=0$. Hence $ f(B)\subseteq \Ker \rho= \beta(B)$; so $\beta(B)=B$ and $\beta$ is an isomorphism.
\end{proof}

\section{Gabriel topologies} \label{S:gab-top}
In this section we briefly introduce Gabriel topologies and discuss some advancements that relate Gabriel topologies to $1$-tilting classes and silting classes over commutative rings as done in \cite{H} and \cite{AHHr}. Furthermore, we include some of our own results Subsection~\ref{SS:gab-prop}. For a more detailed discussion on torsion pairs and Gabriel topologies, refer to \cite[Chapters VI and IX]{Ste75}.

We will start by giving definitions in the case of a general ring with unit (not necessarily commutative). 

Recall that a {\it torsion pair} $(\E, \F)$ in $\ModR$ is a pair of classes of modules in $\ModR$ which are mutually orthogonal with respect to the $\Hom$-functor and maximal with respect to this property.
 The class $\E$ is called a {\it torsion class} and $\F$ a {\it torsion-free class}.

 A class $\C$ of modules is a { torsion class} if and only if it is closed under extensions, direct sums, and epimorphic images. A torsion pair $(\E, \F)$ is called {\it hereditary} if $\E$ is also closed under submodules.
 
 A torsion pair  $(\E, \F)$ is {\it generated} by a class $\C$ if $\F$ consists of all the modules $F$ such that $\Hom_R(C, F)=0$ for every $C\in \C$.

A {\it (right) Gabriel topology} on $R$ is a filter of right ideals of $R$, denoted $\G$, such that the following conditions hold. Recall that for a right ideal $I$ in $R$ and an element $t \in R$, $(I:t) := \{r \in R: tr \in I\}$.
\begin{itemize}
\item[(i)] If $I \in \G$ and $r \in R$ then $(I:r) \in \G$.
\item[(ii)] If $J$ is a right ideal of $R$ and there exists a $I \in \G$ such that $(J:t) \in \G$ for every $t \in I$, then $J \in \G$.
\end{itemize}
Right Gabriel topologies on $R$ are in bijective correspondence with hereditary torsion pairs in $\ModR$. Indeed, to each right Gabriel topology $\G$, one can associate the following hereditary torsion class.
\[
\E_\G = \{ M \mid \Ann_R (x) \in \G \text{ for every } x \in M\}
\]
Then, the corresponding torsion pair $(\E_\G, \F_\G)$ is generated by the cyclic modules $R/J$ where $J \in \G$. The classes $\E_\G$ and $\F_\G$ are referred to as the {\it $\G$-torsion} and {\it $\G$-torsion-free} classes, respectively.

Conversely, if $(\E, \F)$ is a hereditary torsion pair in $\ModR$, the set \[\{J\leq R\mid R/J\in \E\}\] is a right Gabriel topology.

 For a right $R$-module $M$ let $t_\G(M)$ denote the $\G$-torsion submodule of $M$, or sometimes $t(M)$ when the Gabriel topology is clear from context.

The {\it module of quotients} of the Gabriel topology $\G$ of a right $R$-module $M$ is the module
\[
M_\G := \varinjlim_{\substack{J \in \G}} \Hom_R(J, M/t_\G(M)).
\]
Furthermore, there is a canonical homomorphism 
\[\psi_M: M\cong \Hom_R(R, M) \to M_\G.\]
By substituting $M=R$, the assignment gives a ring homomorphism $\psi_R:R \to R_\G$ and furthermore, for each $R$-module $M$ the module $M_\G$ is both an $R$-module and an $R_\G$-module. Both the kernel and cokernel of the map $\psi_M$ are $\G$-torsion $R$-modules, and in fact $\Ker(\psi_M) = t_\G(M)$.

Let $q: \ModR \to \ModR_\G$ denote the functor that maps each $M$ to its module of quotients. Let $\psi^\ast$ denote the right exact functor $\ModR \to \ModR_\G$ where $\psi^\ast(M):= M \otimes R_\G$. In general, there is a natural transformation $\Theta: \psi^\ast \to q$ with $\Theta_M:M\otimes R_\G \to M_\G$ which is defined as $m \otimes \eta \mapsto \psi_M(m) \cdot \eta$. That is, for every $M$ the following triangle commutes.
\[
(\star)\qquad \qquad \xymatrix{
M \ar[rr]^{\psi^\ast(M)}  \ar[dr]_{\psi_M}&&M \otimes_R R_\G \ar[dl]^{\Theta_M}\\
&M_\G&}
\]
A right $R$-module is {\it $\G$-closed} if the following natural homomorphisms are all isomorphisms for every $J \in \G$.
\[
M \cong \Hom_R(R, M) \to \Hom_R (J, M)
\]
This amounts to saying that $\Hom_R(R/J,M) =0$ for every $J \in \G$ (i.e. $M$ is {\it $\G$-torsion-free}) and $\Ext^1_R(R/J,M) =0$ for every $J \in \G$ (i.e. $M$ is {\it $\G$-injective}). Thus if $M$ is $\G$-closed then $M$ is isomorphic to its module of quotients $M_\G$. Conversely, every $R$-module of the form $M_\G$ is $\G$-closed. The $\G$-closed modules form a full subcategory of both $\ModR$ and $\ModR_\G$.

A left $R$-module $N$ is called {\it $\G$-divisible} if $JN = N$ for every $J\in \G$. Equivalently, $N$ is $\G$-divisible if and only if $R/J \otimes_R N =0$ for each $J \in \G$. We denote the class of $\G$-divisible modules by $\D_\G$. It is straightforward to see that $\D_\G$ is a torsion class in $\RMod$.

A right Gabriel topology is {\it faithful} if $\Hom_R(R/J, R) =0$ for every $J \in \G$, or equivalently if  $R$ is $\G$-torsion-free, that is the natural map $\psi_R\colon R \to R_\G$ is injective. 

A right Gabriel topology is {\it finitely generated} if it has a basis consisting of finitely generated right ideals. Equivalently, $\G$ is finitely generated if the $\G$-torsion radical preserves direct limits (that is there is a natural isomorphism $t_\G(\varinjlim_i M_i) \cong \varinjlim_i (t_\G(M_i))$) if and only if the $\G$-torsion-free modules are closed under direct limits (that is, the associated torsion pair is of finite type). The first of these two equivalences was shown in \cite[Proposition XIII.1.2]{Ste75}, while the second was noted by Hrbek in the discussion before \cite[Lemma 2.4]{H}.

\subsection{More properties of Gabriel topologies}\label{SS:gab-prop}
We note that in the following two lemmas, all statements hold in the non-commutative case except for Lemma~\ref{L:G-top-facts} (iii). Otherwise, all the Gabriel topologies will be right Gabriel topologies, therefore the associated torsion pair $(\E_\G, \F_\G)$ are classes of right $R$-modules and the $\G$-divisible modules are left $R$-modules.

We will often refer to the following exact sequence where $\psi_R$ is the ring of quotients homomorphism in (\ref{eq:psiR}) as discussed in the previous subsection. We often will denote $t_\G(M)$ simply by $t(M)$ and when clear from the context, $\psi$ instead of $\psi_R$ and $\bar{\psi}:R/t(R) \to R_\G$ the induced monomorphism from $\psi_R$.
\begin{equation}\label{eq:psiR}
0 \to t_\G(R) \to R \overset{\psi_R} \to R_\G \to R_\G/\psi_R(R) \to 0
\end{equation}

\begin{lem}\label{L:G-top-facts}
Suppose $\G$ is a right Gabriel topology. Then the following statements hold. 
\begin{enumerate}
\item[(i)] If $M$ is a $\G$-torsion (right) $R$-module and $D$ is a $\G$-divisible module then $M \otimes_R D =0$.
\item[(ii)] If $N$ is a $\G$-torsion-free module then the natural map \\$\id_N \otimes_R \psi_R\colon  N \to N \otimes_R R_\G$ is a monomorphism and $N \to N \otimes_R R/t(R)$ is an isomorphism. 
\item[(iii)] Suppose $R$ is commutative. If $D$ is both $\G$-divisible and $\G$-torsion-free, then $D$ is a $R_\G$-module and $D \cong D\otimes_R R_\G$ via the natural map\\ $\id_D \otimes_R \psi_R\colon D \otimes_R R \to D \otimes_R R_\G$.
\item[(iv)]If $X$ is an $R$-$R$-bimodule and is $\G$-torsion, then $M \otimes_R X$ is $\G$-torsion for every $M \in \ModR$.
\end{enumerate}
\end{lem}
\begin{proof}
(i) This is from \cite[Proposition VI.9.1]{Ste75}. Suppose $M$ is a $\G$-torsion module and $D$ is a $\G$-divisible module. Then there is the following surjection.
\[
\bigoplus_{\substack{\alpha \in A \\ J_\alpha \in \G}}R/J_\alpha \to M \to 0 
\]
As $R/J \otimes_R D =0$ for every $J \in \G$ by definition, the conclusion follows by applying $(-\otimes_R D)$ to the above sequence. \\
(ii) Consider the following commuting triangle where $N$ is $\G$-torsion-free in $\ModR$. 
\[
\xymatrix@C=2cm{
0 \ar[d]& \\
N\cong N \otimes_RR \ar[r]^{\id_N \otimes_R \psi_R  \hspace{10pt}}  \ar[d]_{\psi_N}&N \otimes_R R_\G \ar[dl]^{\Theta_N}\\
N_\G&}
\]
Then $\psi_N$ is a monomorphism and since $\psi_N = \Theta_N \circ (\id_N \otimes_R \psi_R)$, also $\id_N \otimes_R \psi_R$ is a monomorphism. Moreover, we know that $\id_N \otimes_R \psi_R$ factors as follows. 
\[N \twoheadrightarrow N \otimes_R R/t(R) \to N \otimes_R R_\G\]
Thus also $N \twoheadrightarrow N \otimes_R R/t(R)$ is a monomorphism, and therefore is an isomorphism.\\
(iii) Consider the following commuting diagram where the horizontal sequence is exact by (i) as $D$ is $\G$-torsion-free.
\[
\xymatrix@C=1.7cm{
&0 \ar[d]& \\
0 \ar[r] &D \ar[r]^{\id_D \otimes_R \psi_R  \hspace{10pt}}  \ar[d]_{\psi_D}&D \otimes_R R_\G \ar[dl]^{\Theta_D} \ar[r] & D \otimes_R R_\G/\psi(R) \ar[r] &0\\
&D_\G&}
\]
Additionally, $D\otimes_R R_\G/\psi(R)=0$, since $R_\G/\psi(R)$ is $\G$-torsion. Therefore the following map is an isomorphism.
 \[\id_D \otimes_R \psi_R\colon D \to D \otimes_R R_\G\]
 \\
(iv) Fix $X$ a $\G$-torsion $R$-$R$-bimodule and $M \in \ModR$. Take a free presentation of $M$, $R^{(\alpha)} \to M \to 0$. Apply $(- \otimes_RX)$ to find the following exact sequence.
\[
X^{(\alpha)} \to M \otimes_R X \to 0
\]
As $X^{(\alpha)}$ is $\G$-torsion and the $\G$-torsion modules are closed under quotients, also $M \otimes_R X$ is $\G$-torsion.
\end{proof}
\begin{lem} \label{L:tor-Rg}
Suppose $\G$ is a Gabriel topology of right ideals. Then the following hold.
\begin{enumerate}
\item[(i)] If $\pdim M_R \leq 1$, then $\Tor^R_1(M, R_\G)=0$. 
\item[(ii)] If $\pdim M_R \leq 1$ and $M$ is $\G$-torsion-free, then \\$\Tor^R_1(M, R_\G)=0=\Tor^R_1(M, R_\G/\psi(R))$. 
\end{enumerate}
If moreover $\G$ is a right Gabriel topology with a basis of finitely generated ideals and $\pdim M_R \leq 1$, then $M_R \otimes_R R_\G$ is $\G$-torsion-free.
\end{lem}
\begin{proof}
(i)
By assumption $\pdim M_R \leq 1$, so there is the following projective resolution of $M$, where $P_0, P_1$ are projective right $R$-modules.
\begin{equation}\label{E:fan1}
0 \to P_1 \overset{\gamma}\to P_0 \to M \to 0
\end{equation}
We will first show that $\Tor^R_1(M, R/t(R))=0$. We first note that from the following exact sequence (\ref{E:fan4}), $\Tor^R_1(M, R/t(R))$ is $\G$-torsion, as it is contained in the $\G$-torsion module $M \otimes_R t(R)$ (see Lemma~\ref{L:G-top-facts}(iv)) and is itself a right $R$-module as $R/t(R)$ is an $R$-$R$-bimodule.
\begin{equation}\label{E:fan4}
0 \to \Tor^R_1(M, R/t(R)) \to M \otimes_R t(R) \to M \to M \otimes_R R/t(R) \to 0
\end{equation}
Next, we note that from the following exact sequence (\ref{E:fan2}), $\Tor^R_1(M, R_\G)$ is $\G$-torsion-free as it is contained in the $\G$-torsion-free module $P_1 \otimes_R R_\G$. 
\begin{equation}\label{E:fan2}
0 \to \Tor^R_1(M, R_\G) \to P_1 \otimes_R R_\G \to P_0 \otimes_R R_\G \to M \otimes_R R_\G \to 0
\end{equation}
Thus from the following exact sequence (\ref{E:fan3}), $\Tor^R_1(M, R/t(R))$ is $\G$-torsion-free as by assumption $\Tor^R_2(M, R_\G/\psi(R))=0$. Therefore we conclude that $\Tor^R_1(M, R/t(R))=0$ as it is both $\G$-torsion and $\G$-torsion-free.
\begin{equation}\label{E:fan3}
0 \to \Tor^R_1(M, R/t(R)) \to \Tor^R_1(M, R_\G) \to \Tor^R_1(M, R_\G/\psi(R)) 
\end{equation}
Moreover, also $\Tor^R_1(M, R_\G/\psi(R))$ is $\G$-torsion by applying $(- \otimes_R R_\G/\psi(R))$ to the short exact sequence (\ref{E:fan1}). Therefore $\Tor^R_1(M, R_\G)=0$ as it is both $\G$-torsion by (\ref{E:fan3}) and $\G$-torsion-free by (\ref{E:fan2}). 

(ii) 
 By applying the functor $(M \otimes_R -)$ to the exact sequence $0 \to R/t(R) \overset{\bar{\psi}}\to R_\G \to R_\G/\psi(R_\G) \to 0$, we have the following exact sequence.
 \[
\xymatrix{
0= \Tor^R_1(M, R_\G) \ar[r] & \Tor^R_1( M, R_\G/\psi(R)) \ar[r] & M \cong M \otimes_R R/t(R)  }
\]
By Lemma~\ref{L:G-top-facts}(ii), $\id_M \otimes_R \bar{\psi_R}\colon M \otimes_RR/t(R) \to M \otimes_R R_\G $ is a monomorphism and by (i), $\Tor^R_1(M, R_\G)=0$ hence $\Tor^R_1(M, R_\G/\psi(R))=0$.

For the final statement, first note that for any projective right $R$-module $P_R$, $P_R \otimes_R R_\G \underset{\oplus} \leq R_\G^{(\alpha)}$. By the assumption that $\G$ is finitely generated, by \cite[Proposition XIII.1.2]{Ste75}, we have that arbitrary direct sums of copies of $\G$-closed modules are $\G$-closed, thus we conclude that $P_R \otimes_R R_\G$ is $\G$-closed. Now consider the presentation $0 \to P_1 \to P_0 \to M \to 0$ of $M$ with $P_0, P_1$ projective. Then $0 \to P_1\otimes_R R_\G \to P_0\otimes_R R_\G \to M\otimes_R R_\G \to 0$ is exact as $\Tor^R_1(M, R_\G)=0$ by (i) of this lemma. As the middle term $P_0\otimes_R R_\G$ is $\G$-torsion-free and $P_1 \otimes_R R_\G$ is $\G$-closed, it follows that $M \otimes_R R_\G$ is $\G$-torsion-free.  

\end{proof}

\subsection{Gabriel topologies and $1$-tilting classes}\label{SS:gab-tilt}

In this paper, we will only be concerned with Gabriel topologies over commutative rings. In this setting, much useful research has already done in this direction. Specifically, in \cite{H}, Hrbek showed that over commutative rings the faithful finitely generated Gabriel topologies are in bijective correspondence with $1$-tilting classes, and that the latter are exactly the classes of $\G$-divisible modules for some faithful finitely generated Gabriel topology $\G$, as stated in the following theorem. 

\begin{thm} \cite[Theorem 3.16]{H} \label{T:Hrb-tilting}
Let R be a commutative ring. There are bijections between the following collections.
\begin{enumerate}
\item[(i)] 1-tilting classes $\T$ in $\ModR$. 
\item[(ii)] faithful finitely generated Gabriel topologies $\G$ on $R$.
\item[(iii)] faithful hereditary torsion pairs $(\E,\F)$ of finite type in $\ModR$.
\end{enumerate}
Moreover, the tilting class  $\T$ is the class of $\G$-divisible modules with respect to the Gabriel topology $\G$.
\end{thm}

When we refer to the {\it Gabriel topology associated to the $1$-tilting class $\T$} we will always mean the Gabriel topology in the sense of the above theorem. In addition we will often denote $\A$ to be the right $\Ext$-orthogonal class to $\D_\G =\T$ in the situation just described, so $(\A, \D_\G)$ will denote the $1$-tilting cotorsion pair.

%

In \cite{AHHr} the correspondence between faithfully finitely generated Gabriel topologies and $1$-tilting classes over commutative rings was extended to finitely generated Gabriel topologies which were shown to be in bijective correspondence with silting classes. Thus in this case the class $\D_\G$ of $\G$-divisible modules coincides with the class $\Gen T$ for some silting module $T$.

\subsection{Homological ring epimorphisms}\label{SS:homological}
There is a special class of Gabriel topologies which behave particularly well and are related to ring epimorphisms. The majority of this paper will be restricted to looking at these Gabriel topologies. 
The standard examples of these Gabriel topologies over $R$ are localisations of a commutative ring $R$ with respect to a multiplicative subset $S$, where the Gabriel topology has as a basis the principal ideals generated by elements of $S$.

A {\it ring epimorphism} is a ring homomorphism $R \overset{u}\to U$ such that $u$ is an epimorphism in the category of unital rings. This is equivalent to the natural map $ U \otimes_R U \to U$ induced by the multiplication in $U$ being an isomorphism, or equivalently that $U \otimes_R (U / u(R)) =0$ (see \cite[Chapter XI.1]{Ste75}.

Two ring epimorphisms $R \overset{u}\to U$ and $R \overset{u'}\to U'$ are equivalent if there is a ring isomorphism $\sigma\colon U\to U'$ such that $\sigma  u=u'$.

A ring epimorphism is {\it homological} if $\Tor^R_n(U_R,{}_RU) = 0$ for all $n >0$. A ring epimorphism is called {\it (left) flat} if $u$ makes $U$ into a flat left $R$-module. Clearly all flat ring epimorphisms are homological. We will denote the cokernel of $u$ by $K$ and sometimes by $U/R$ or $U/u(R)$.

A left flat ring epimorphism $R \overset{u}\to U$ is called a {\it perfect right localisation} of $R$. In this case, by \cite[Chapter XI.2, Theorem 2.1]{Ste75} the family of right ideals
\[
\G = \{ J \leq R \mid J U = U  \} 
\]
forms a right Gabriel topology. Moreover, there is a ring isomorphism $\sigma:U \to R_\G$ such that $\sigma  u: R \to R_\G$ is the canonical isomorphism $\psi_R: R \to R_\G$, or, in other words, $u$ and $\psi_R$ are equivalent ring epimorphisms. Note also that a right ideal $J$ of $R$ is in $\G$ if and only if $R/J \otimes_R U =0$.

We will make use of the characterisations of perfect right localisations from Proposition 3.4 in Chapter XI.3 of Stenstr\"om's book \cite{Ste75}. 

In particular, Proposition 3.4 states that the right Gabriel topology $\G$ associated to a flat ring epimorphism $R \overset{u}\to U$ is finitely generated and the $\G$-torsion submodule $t_\G(M)$ of a right $R$-module $M$ is the kernel of the canonical homomorphism $M\to M \otimes_R U $. Thus, $K=U/u(R)$ is $\G$-torsion, hence $\Hom_R(K, U)=0$.
If moreover the flat ring epimorphism $R \overset{u}\to U$ is injective, then $ \Tor^R_1(M, K) \cong t_\G(M)$ and $\G$ is faithful.

\begin{rem}\label{R:pdU=1}  \emph{From the above observations and results in \cite{H}, when $R$ is commutative and $R \overset{u}\to U$ is a flat injective epimorphism one can associate a $1$-tilting class which is exactly the class of $\G$-divisible modules.
In the case that additionally $\pdim_R U \leq 1$, one can apply a result from \cite{AS} which states that $U \oplus K$ is a $1$-tilting module, so there is a $1$-tilting class denoted $\T: =(U \oplus K)^\perp = \Gen(U)$. In fact, we claim that this is exactly the $1$-tilting class of $\G$-divisible modules. Explicitly, the Gabriel topology associated to $\T$ in the sense of Theorem~\ref{T:Hrb-tilting} is exactly the collection of ideals $\{J \mid JM = M \text{ for every } M \in \T \}$. The Gabriel topology that arises from the perfect localisation is the collection $\{J \mid JU = U \}$ and since $U \in \T = \Gen U$, the Gabriel topologies associated to these two $1$-tilting classes are the same. We conclude that the two $1$-tilting classes coincide: $\Gen_R(U) = \D_\G$.
\\
 In \cite[Proposition 5.4]{H} the converse is proved: If one starts with a $1$-tilting class $\T$ with associated Gabriel topology $\G$, so that $\T=\D_\G$, then $R_\G$ is a perfect localisation and $\pdim R_\G \leq 1$ if and only if $\Gen R_\G  = \D_\G$.} 
 \end{rem}

The following lemma will be useful when working with a Gabriel topology over a commutative ring that arises from a perfect localisation.
 \begin{lem} \label{L:finmanyann}
Let $R$ be a commutative ring, $u:R \to U$ a flat injective ring epimorphism, and $\G$ the associated Gabriel topology. Then the annihilators of the elements of $U/R$ form a sub-basis for the Gabriel topology $\G$. That is, for every $J\in \G$ there exist $z_1, z_2, \dots , z_n \in U$ such that 
\[
\bigcap_{\substack{
   0 \leq i \leq n}}
   \Ann_R(z_i +R) \subseteq J.\] 
 \end{lem}
 \begin{proof} Every ideal of the form $\Ann_R(z+R)$ is an ideal in $\G$ since $K=U/R$ is $\G$-torsion.
 
 Fix an ideal $J \in \G$. Then, $U = JU$, so $1_U = \sum_{0 \leq i \leq n} a_i z_i$ where $a_i \in J$ and $z_i \in U$. We claim that 
 \[
\bigcap_{\substack{
   0 \leq i \leq n}}
   \Ann_R(z_i +R) \subseteq J.
\] 
Take $b \in \bigcap_{\substack{
   0 \leq i \leq n}}
   \Ann_R(z_i +R)$. Then 
 \[
 b = \sum_{0 \leq i \leq n} b a_i z_i \in J
 \]
 since each $b z_i \in R$, hence $b a_i z_i \in J$, and it follows that $b \in J$. 
 \end{proof}

\section{Enveloping $1$-tilting classes over commutative rings}\label{S:tilting-enveloping}
For this section, $R$ will always be a commutative ring and $\T$ a $1$-tilting class. 

By Theorem~\ref{T:Hrb-tilting} there is a faithful finitely generated Gabriel topology $\G$ such that  $\T$ is the class of $\G$-divisible modules.
We denote again by  $(\E_\G, \F_\G)$ the associated faithful hereditary torsion pair of finite type. We use $\D_\G$ and $\T = \Gen T = T^\perp$ interchangeably to denote the $1$-tilting class, and $\A$ to denote the right orthogonal class ${}^\perp \D_\G$.
\\
The aim of this section is to show that if $\T$ is enveloping, then $R_\G$, the ring of quotients with respect to $\G$, is $\G$-divisible and therefore $\psi_R:R \to R_\G$ is a perfect localisation of $R$.

 Recall that if $\T$ is $1$-tilting, $\T \cap {}^\perp\T = \Add(T)$ (see \cite[Lemma 13.10]{GT12}). By (T3) of the definition of a $1$-tilting module we have the following short exact sequence \[ \text{(T3) }\quad
0 \to R \overset{\varepsilon}\to T_0 \to T_1 \to 0
\] where $T_0, T_1 \in \Add(T)$. In fact, this short exact sequence is a special $\D_\G$-preenvelope of $R$, and $T_0 \oplus T_1$ is a $1$-tilting module which generates $\T$ by \cite[Theorem 13.18 and Remark 13.19]{GT12}. 

Furthermore, assuming that $R$ has a $\D_\G$-envelope, we  can suppose without loss of generality that the sequence (T3) is the $\D_\G$-envelope of $R$, since an envelope is extracted from a special preenvelope by passing to direct summands (Proposition~\ref{P:Xu-env}). 
For the rest of the section we will denote the $\D_\G$-envelope of $R$ by $\varepsilon$. 


Recall from Section~\ref{S:gab-top} that 
for every $M \in \ModR$ there is the commuting diagram $(\star)$. 
Since $\G$ is faithful we have the following short exact sequence where $\psi_R$ is a ring homomorphism and $R_\G/R$ is $\G$-torsion.
\[(\dag)\quad
0 \to R \overset{\psi_R} \to R_\G \to R_\G/R \to 0
\]

 We now show two lemmas about the $1$-tilting module $T_0 \oplus T_1$ and the class $\Add (T_0 \oplus T_1)$ assuming that $R$ has a $\D_\G$-envelope.
\begin{lem} \label{L:R-env}
Let the following short exact sequence be a $\D_\G$-envelope of $R$.
\[
0 \to R \overset{\varepsilon}\to T_0 \to T_1 \to 0
\]
Then $T_0$ is $\G$-torsion-free and $T_0 \cong T_0 \otimes_R R_\G$.
\end{lem}
\begin{proof}
We will show that  for every $J \in \G$, $T_0 [J]$, the set of elements of $T_0$ annihilated by $J$ is zero. Set $w:=\varepsilon(1_R)$ and fix a $J \in \G$. As $T_0 = JT_0$, $w = \sum_{1 \leq i \leq n} a_i z_i$ where $a_i \in J$ and $z_i \in T_0$. This sum is finite, so we can define the following maps.
\[ \xymatrix@R=.1cm{
{\mathbf{z}}:R \ar[r] & \bigoplus_{1 \leq i \leq n} T_0 & {\mathbf{a}}:\bigoplus_{1 \leq i \leq n} T_0 \ar[r] & T_0\\
\hspace{15pt}1_R \ar@{|-_{>}}[r] &(z_1, ..., z_n) & \hspace{5pt}(x_1, ..., x_n) \ar@{|-_{>}}[r] & \sum_i a_ix_i}
\]
As $\bigoplus_nT_0$ is also $\G$-divisible, by the preenvelope property of $\varepsilon$ there exists a map $f:T_0 \to \bigoplus_nT_0$ such that $f \varepsilon = {\mathbf{z}}$. Also, ${\mathbf{a}}{\mathbf{z}}(1_R) = \sum_{1 \leq i \leq n} a_i z_i = w$, so ${\mathbf{a}}{\mathbf{z}} = \varepsilon$ and the following diagram commutes.
\[\xymatrix@C=1.8cm@R=1.3cm{
0\ar[r]&
{R}\ar[dr]^{{\mathbf{z}}} \ar[ddr]_\varepsilon \ar[r]^\varepsilon&T_0\ar[r]^{\beta} \ar[d]^{f}& T_1  \ar[r]&0\\
& & \bigoplus_n T_0 \ar^{{\mathbf{a}}}[d]\\
& & T_0}
\]
By the envelope property of $\varepsilon$, ${\mathbf{a}}f$ is an automorphism of $T_0$. The restriction of the automorphism ${\mathbf{a}}f$ to $T_0[J]$ is an automorphism of $T_0[J]$, and factors through the module $\bigoplus_nT_0[J]$. However $\mathbf{a}( \bigoplus_nT_0[J]) =0$, so ${\mathbf{a}}f(T_0[J]) =0$, but $\mathbf{a}f$ restricted to $T_0[J]$ is an automorphism, thus $T_0[J] =0$.
\\
From (iii) of Lemma~\ref{L:G-top-facts} it follows that $T_0 \cong T_0 \otimes_R R_\G$ since $T_0$ is $\G$-divisible.
\end{proof}

We look at $\D_\G$-envelopes of $\G$-torsion modules in $\ModR$, and find that they are also $\G$-torsion.

\begin{lem} \label{L:torsion-env}
Suppose $\D_\G$ is enveloping in $\ModR$ and $M$ is a $\G$-torsion $R$-module. Then the $\D_\G $-envelope of $M$ is $\G$-torsion.
\end{lem}
\begin{proof}
To begin with, fix a finitely generated $J \in \G$ with a set $\{a_1, \dots, a_t\}$ of generators and consider a $\D_\G $-envelope  $D(J)$ of the cyclic $\G$-torsion module $R/J$, denoted as follows. 
\[
0 \to R/J \hookrightarrow D(J) \to A(J) \to 0
\]
We will show that $D(J)$ is $\G$-torsion.
Consider the following countable direct sum of envelopes of $R/J$ which is itself an envelope, by Theorem~\ref{T:Xu-sums}~(i).
 \[ 0 \to \bigoplus_{n}
(R/J)_{n}\hookrightarrow \bigoplus_{n} D(J)_{n} \to \bigoplus_{n}A(J)_{n} \to 0 .\] 
Choose an element $a \in J$ and for each $n$ set $f_n\colon D(J)_n\to D(J)_{n+1}$ to be the multiplication by $a$.
  
Then clearly $(R/J)_n$ vanishes under the action of $f_n $, hence we can apply Theorem~\ref{T:Xu-sums}~(ii). For every $d \in D(J)$, there exists an $m$ such that 
\[ f_m \circ \cdots \circ f_2 \circ f_1 (d) = 0 \in D(J)_{(m+1)}.\]
Hence for every $d \in D$ there is an integer $m$ for which $a^m d = 0$.
 \\
Fix  $d \in D$ and let $m_i$ be the minimal natural number for which $(a_i)^{m_i}d=0$ and set $m:= \sup\{m_i: 1 \leq i \leq t\}$. Then for a large enough integer $k$ we have that $J^k d =0$  (for example set $k=tm$), and $J^k \in \G$. Thus every element of $D(J)$ is annihilated by an ideal contained in $\G$, therefore $D(J)$ is $\G$-torsion.

Now consider an arbitrary $\G$-torsion module $M$. Then $M$ has a presentation $\bigoplus_{\alpha\in \Lambda} R/J_\alpha \overset{p}\to M \to 0$ for a family 
$\{J_\alpha\}_{\alpha \in \Lambda}$ of ideals of $\G$. Since $\G$ is of finite type, we may assume that each $J_{\alpha}$ is  finitely generated.

Take the push-out of this map with the $\D_\G$-envelope of $\bigoplus_\alpha R/J_\alpha$.
\[\xymatrix{
0 \ar[r]&\bigoplus_{\substack{\alpha \in \Lambda}} R/J_\alpha \ar[r] \ar[d]^p& \bigoplus_{\substack{\alpha \in \Lambda}}D(J_\alpha) \ar[r] \ar[d]& \bigoplus_{\substack{\alpha \in \Lambda}}A(J_\alpha) \ar@{=}[d] \ar[r] & 0\\
0 \ar[r]&M\ar[r] \ar[d] & Z \ar[r] \ar[d]& \bigoplus_{\substack{\alpha \in \Lambda}}A(J_\alpha) \ar[r] & 0\\
&0&0 &&}
\]
The bottom short exact sequence forms a preenvelope of $M$. We have shown above that for every $\alpha$ in $A$, $D(J_\alpha)$ is $\G$-torsion, so also $Z$ is $\G$-torsion. Therefore, as the $\D_\G$-envelope of $M$ must be a direct summand of $Z$ by Proposition~\ref{P:Xu-env}, also the $\D_\G$-envelope of $M$ is $\G$-torsion.
\end{proof}
The following is a corollary to the final statement of Lemma~\ref{L:tor-Rg}
and Lemma~\ref{L:torsion-env}. 
\begin{cor} \label{C:tor-tens-div}
Suppose $\D_\G$ is enveloping in $\ModR$ and suppose $M$ is a $\G$-torsion $R$-module. Then $M \otimes_R R_\G$ is $\G$-divisible. 
\end{cor}
\begin{proof}
Let the following be a $\D_\G$-envelope of a $\G$-torsion module $M$, where both $D$ and $A$ are $\G$-torsion by Lemma~\ref{L:torsion-env}.
\[
0 \to M \to D \to A \to 0
\]
The module $A$ is $\G$-divisible and $R_\G/R$ is $\G$-torsion so $A \otimes_R R_\G/R =0$, hence $A \to A \otimes_R R_\G$ is surjective. In particular, $A \otimes_R R_\G $ is $\G$-torsion. Also as $A \in \Add (T_0 \oplus T_1)$, $A \otimes_R R_\G$ is $\G$-torsion-free by Lemma~\ref{L:tor-Rg}. It follows that $A\otimes_R R_\G$ is both $\G$-torsion and $\G$-torsion-free so $A\otimes_R R_\G=0$. Additionally as $\pdim A \leq 1$, $\Tor^R_1(A, R_\G)=0$, so the functor $(- \otimes_R R_\G)$ applied to the envelope of $M$ reduces to the following isomorphism. 
\[
0=\Tor^R_1(A, R_\G) \to M \otimes_R R_\G \overset{\cong}\to D \otimes_R R_\G \to A \otimes_R R_\G =0
\]
Hence as $D \otimes_R R_\G$ is $\G$-divisible, also $M \otimes_R R_\G$ is $\G$-divisible, as required.
\end{proof}

\begin{prop}\label{P:R_G-divisible}
Suppose $\D_\G$ is enveloping in $\ModR$. Then $R_\G$ is $\G$-divisible.
\end{prop}
\begin{proof}
We will show that for each $J \in \G$, $R/J \otimes_R R_\G =0$. Fix a $J \in \G$. By Corollary~\ref{C:tor-tens-div}, $R/J \otimes_R R_\G$ is $\G$-divisible, Thus we have $R/J \otimes_R (R/J \otimes_R R_\G) =0$. However \[0=R/J \otimes_R (R/J \otimes_R R_\G) \cong (R/J \otimes_R R/J) \otimes_R R_\G\cong R/J\otimes_RR_\G,\] since $R\to R/J$ is a ring epimorphism, thus $R_\G$ is $\G$-divisible.
\end{proof}

Using the characterisation of a perfect localisation of \cite[Chapter XI.3, Proposition 3.4]{Ste75}, we can state the main result of this section.

\begin{prop} \label{P:tilting-env}
Assume that $\T$ is a $1$-tilting class over a commutative ring $R$ such that the class $\T$ is enveloping. Then the associated Gabriel topology $\G$ of $\T$ arises from a perfect localisation.
\end{prop}
\begin{proof}
By Proposition~\ref{P:R_G-divisible}, $R_\G$ is $\G$-divisible, hence by \cite[Proposition 3.4~(g)]{Ste75}, $\psi\colon R\to R_\G$ is flat ring epimorphism and moreover it is injective.
\end{proof}

 \section{The $\G$-completion of $R$ and the endomorphism ring of $K$}\label{S:compl-endK}
The aim of this section is to prove that if $R \overset{u}\to U$ is a commutative flat injective ring epimorphism  with associated Gabriel topology $\G$, then there is a natural ring isomorphism between the following two rings.
\[
\Lambda(R) = \varprojlim_{\substack{J \in \G}} R/J {\rm \ and\ } \End_R(K)=\mathfrak{R}
\]
This was mentioned in \cite[Remark 19.4]{BP2}, and a much stronger equivalence was shown in \cite{Pos3}. Also, it follows from this ring isomorphism that $\mathfrak{R}$ is a commutative ring.

For completeness, we will give an explicit description of the isomorphism between the two rings.

We will begin by briefly recalling some useful definitions about topological rings specifically referring to Gabriel topologies. Our reference is \cite[Chapter VI.4]{Ste75}. Next we will continue by introducing $u$-contramodules in an analogous way to Positselski in \cite{Pos}. To finish, we show the ring isomorphism as well as a lemma and a proposition which relate the $\G$-torsion $R$-modules $R/J$ to the discrete quotient rings of $\mathfrak{R}$.

 \subsection{Topological rings}
A ring $R$ is a {\it topological ring} if it has a topology such that the ring operations are continuous.

A topological ring $R$ is {\it right linearly topological} if it has a topology with a basis of neighbourhoods of zero consisting of right ideals of $R$. The ring $R$ with a right Gabriel topology is an example of a right linearly topological ring.

If $R$ is a right linearly topological ring, then the set of right ideals $J$ in a basis  $\mathfrak{ B}$ of the topology form a directed set, hence $\{R/J\mid J\in 
\mathfrak B\}$ is an inverse system. 
 The {\it completion} of $R$ is the module
\[
\Lambda_{\mathfrak{ B}}(R) := \varprojlim_{\substack{J \in \mathfrak B}} R/J.
 \]
There is a canonical map $\lambda:R \to \Lambda_\mathfrak{B}(R)$ which sends the element $r\in R$ to $(r +J)_{J\in \mathfrak{ B}}$. If the homomorphism $\lambda_R$ is injective, then $R$ is called {\it separated}, which is equivalent to $\bigcap_{J \in \mathfrak{ B}}J =0$. If the map $\lambda$ is surjective, $R$ is called {\it complete}. 
\\
The {\it projective limit topology} on $\Lambda_{\mathfrak{ B}}(R)$ is the topology where a sub-basis of neighbourhoods of zero is given by the the kernels of the projection maps $\Lambda_{\mathfrak{ B}}(R) \to R/J$. That is, it is the topology induced by the product of the discrete topology on $\prod_{J \in \mathfrak{ B}} R/J$.
 If the ideals in $\mathfrak{ B}$ are two-sided in $R$, then the module $\Lambda_{\mathfrak{ B}}(R)$ is a ring. Furthermore, it is a linearly topological  ring with respect to the projective limit topology. In this case, the ring $\Lambda_{\mathfrak{ B}}(R)$ is both separated and complete with this topology. Each element in $\Lambda_{\mathfrak{ B}}(R)$ is of the form $(r_J +J)_{J\in \mathfrak{ B}}$ with the relation that for $J \subseteq J'$, $r_J - r_{J'} \in J'$. We will simply write $\Lambda(R)$ when the basis $\mathfrak{ B}$ is clear from the context.
\begin{rem}\label{R:topologies} \emph{If $W(J)$ is the kernel of the projection $\pi_J\colon\Lambda_{\mathfrak{ B}}(R)\to R/J$, then clearly $W(J)\supseteq \Lambda(R)J$.}
\end{rem}

Let $R$ be a linearly topological ring. A right  $R$-module $N$ is {\it discrete} if for every $x \in N$, the annihilator ideal $\Ann_R(x) = \{r \in R: xr =0\}$ is open in the topology of $R$. In case the topology on $R$ is a Gabriel topology $\G$, then $N$ is discrete if and only if it is $\G$-torsion.

A linearly topological ring is {\it left pro-perfect} (\cite{Pos5}) if it is separated, complete, and with a base of neighbourhoods of zero formed by two-sided ideals such that all of its discrete quotient rings are perfect.

\emph{ For the rest of this subsection, we will be considering a flat injective ring epimorphism of commutative rings denoted $0 \to R \overset{u} \to U$, and we will denote by $K$ the cokernel $U/R$ of $u$.}

Let $\mathfrak{R}$ denote the endomorphism ring $\End_R(K)$. Take a finitely generated submodule $F$ of $K$, and consider the ideal formed by the elements of $\mathfrak{R}$ which annihilate $F$. The ideals of this form form a base of neighbourhoods of zero of $\mathfrak{R}$. Note that this is the same as considering $\End_R(K)$ with the subspace topology of the product topology on $K^K$ where the topology on $K$ is the discrete topology. We will consider $\mathfrak{R}$ endowed with this topology, which is also called the {\it finite topology}. 

We will now state the above in terms of a Gabriel topology that arises from a perfect localisation. Let $\G$ be the Gabriel topology associated to the flat ring epimorphism $u$. As $K\otimes_RU=0$, $K$ is $\G$-torsion, or equivalently a discrete module. Thus there is a natural well-defined action of $\Lambda(R)$ on $K$. In other words, $K$ is a $\Lambda(R)$-module where for every element $(r_J +J)_{J \in \G}\in \Lambda (R)$ and every element $z\in U$, the scalar multiplication is defined by $(r_J +J)_{J \in \G} \cdot (z+R):= r_{J_z} z +R$ where $J_z:= \Ann_R(z+R)$. 
As well as the natural map  $\lambda:R \to \Lambda(R)$, there is also a natural map $\nu:R \to \mathfrak{R}$ where each element of $R$ is mapped to the endomorphism of $K$ which is multiplication by that element.

If $R \overset{u}\to U$ is a flat injective ring epimorphism, then there is a homomorphism
\[
\alpha: \Lambda(R) = \varprojlim_{\substack{J \in \G}} R/J \to \mathfrak{R},
\]
where $\alpha$ is induced by the action of $\Lambda(R)$ on $K$. It follows that the following triangle commutes.
\[(\ast)\quad  \xymatrix{
R \ar[d]_{\nu}  \ar[r]^\lambda& \Lambda(R) \ar[dl]^{\alpha} \\
\mathfrak{R} & }
\]
The rest of this section is dedicated to showing that $\alpha$ is a ring isomorphism.
We will first show that $\alpha$ is injective, but before that we have to recall some terminology.

%
%

 A module $M$ is {\it $u$-h-divisible} if $M$ is an epimorphic image of $U^{(\alpha)}$ for some cardinal $\alpha$. An $R$-module $M$ has a unique $u$-h-divisible submodule denoted $h_u(M)$, and it is the image of the map $\Hom_R(U,M) \to \Hom_R(R,M) \cong M$. Hence for an $R$-module $M$, by applying the contravariant functor $\Hom_R(-,M)$ to the short exact sequence $0 \to R \overset{u}\to U \to K \to 0$ we have the following short exact sequences.
\begin{equation}\label{eq:1contra}
0 \to \Hom_R(K, M) \to \Hom_R(U,M) \to h_u(M) \to 0 
\end{equation}
\begin{equation} \label{eq:2contra}
0 \to M / h_u(M) \to \Ext^1_R(K,M) \to \Ext^1_R(U,M) \to 0
\end{equation}

By applying the covariant functor $\Hom_R(K,-)$ to the same short exact sequence we have the following.
\begin{equation}\label{eq:3contra}
0= \Hom_R(K, U) \to \Hom_R(K,K) \overset{\delta}\to \Ext^1_R(K, R) \to \Ext^1_R(K,U) =0,  \
\end{equation}
where the last term vanishes since by the flatness of the ring $U$, there is an isomorphism $\Ext^1_R(K,U)\cong \Ext^1_U(K\otimes_RU,U) =0$. Thus note that $\Hom_R(K, K)$ is isomorphic to $\Ext^1_R(K, R)$ via $\delta$.

Recall from Lemma~\ref{L:finmanyann} that the ideals $\Ann_R(z +R)$ for $z+R \in K$ form a sub-basis of the topology $\G$. Let $\clS \subset \G$ denote the ideals of $\G$ of the form $\Ann_R(z +R)$ for $z+R \in K$.  Clearly, the following two intersections of ideals coincide.
\[
\bigcap_{\substack{J \in \G}} J = \bigcap_{\substack{J \in \clS }}  J
\] 
We begin with some facts about $\Lambda(R)$ and $\mathfrak{R}$.

\begin{lem} \label{L:mapfacts}
Let $u:R \to U$ be a flat injective ring epimorphism. Then the following hold.
\begin{enumerate}
\item[(i)] The kernel of $\nu: R \to \mathfrak{R}$ is the intersection $\bigcap_{J \in \clS} J$.
\item[(ii)] The kernel of $\lambda: R \to \Lambda(R)$ is the intersection $\bigcap_{J \in \G} J$.
\item[(iii)] The ideal $\bigcap_{J \in \G} J$ is the maximal $u$-h-divisible submodule of $R$.
\item[(iv)] The homomorphism $\alpha: \Lambda(R) \to \mathfrak{R}$ is injective.
\end{enumerate}
\end{lem}
\begin{proof}
(i) For $r \in R$, $\nu (r)=0$ if and only if $rK=0$.

(ii) By the definition of $\lambda$ it is clear that $\lambda (r) =0$ if and only if $r \in J$ for every $J \in \G$.  

(iii) First we show that $\bigcap_{J \in \G} J \subseteq h_u(R)$. Take $a \in \bigcap_{J \in \G} J$.
 We want to see that multiplication by $a$, $\dot{a}: R \to R$ extends to a map $f:U \to R$ (that is $\dot{a}$ is in the image of the map $u^\ast: \Hom_R(U, R) \to \Hom_R(R,R)$). By part (i) and its proof, $az \in R$ for every $z \in U$, so we have a well-defined map $\dot{a}: U \to R$, which makes the following triangle commute as desired.
\[ \xymatrix{
R \ar[d]_{\dot{a}}  \ar[r]^u&U \ar[ld]^{\dot{a}} \\
R &}
\]
Now take $a \in h_u(R)$. Since $h_u(R)$ is a $\G$-divisible submodule of $R$, $a \in J( h_u(R))\leq J$ for each $J \in \G$, as required.

(iv) Take $\eta=(r_J + J)_{J \in \G} \in \Lambda(R)$ such that $\alpha(\eta)=0$ or $\eta(z+R)=0$ for each $z \in U$. Then $r_Iz \in R$ where $I = \Ann_R(z+R)$. By Lemma~\ref{L:finmanyann}, for each $J \in \G$ there exists $z_0, \dots, z_n$ such that $J \supseteq \bigcap_n \Ann_R(z_i +R) =:I_0$. Thus $r_J - r_{I_0} \in J$ and $r_{I_0}z_i \in R$ for each $0 \leq i \leq n$, so $r_{I_0} \in J$.
This implies $r_J \in J$ for each $J \in \clS$, so $\eta=0$.
\end{proof}
\subsection{$u$-contramodules}\label{S:U-contra}
We will begin by discussing a general commutative ring epimorphism $u$ before moving onto a flat injective ring epimorphisms.
\begin{defn}
Let $u\colon R\to U$ be a ring epimorphism. A {\it $u$-contramodule} is an $R$-module $M$ such that 
\[ \Hom_R(U, M) = 0 = \Ext^1_R(U,M).
\]
 \end{defn}
\begin{lem} \label{L:geiglenz}\cite[Proposition 1.1]{GL91}
 The category of $u$-contramodules is closed under kernels of morphisms, extensions, infinite products and projective limits in $\RMod$.
 \end{lem} 
 The following two lemmas are proved in \cite {Pos} for the case of the localisation of $R$ at a multiplicative subset. The proofs follow analagously for the case of a commutative injective ring epimorphism $u \colon R \to U$. 
 \begin{lem}\label{L:pos1.2}\cite[Lemma 1.2]{Pos}
Let $u\colon R\to U$ be a ring epimorphism and let  $M$ be an $R$-module.
\begin{enumerate}
\item[(i)] If $\Hom_R(U,M)=0$, then $\Hom_R(Z,M)=0$ for any $u$-h-divisible module $Z$.
\item[(ii)] If $M$ is a $u$-contramodule, then $\Ext_R^1(Z,M)=0=\Hom_R(Z,M)$ for any $U$-module $Z$.
\end{enumerate}
 \end{lem}
%
%
%

 \begin{lem}\label{L:oneten} \cite[Lemma 1.10]{Pos}
 Let $b:A \to B$ and $c:A \to C$ be two $R$-module homomorphisms such that $C$ is a $u$-contramodule while $\Ker(b)$ is a $u$-h-divisible $R$-module and $\Coker(b)$ is a $U$-module. Then there exists a unique homomorphism $f:B \to C$ such that $c =fb$.
\end{lem}

From now on, $u: R \to U$ will always be a commutative flat injective ring epimorphism.

The following lemma is proved in two papers of Bazzoni-Positselski, although the proof without spectral sequences is proved in \cite[Lemma 16.2]{BP2}.

\begin{lem} \label{L:endcontra}\cite[Lemma 16.2]{BP2}, \cite[Lemma 2.5(a),(b)]{BP3}
Let $u:R \to U$ be a flat injective ring epimorphism. Then $\mathfrak{R}$ is a $u$-contramodule and is $\G$-torsion-free.
\end{lem}

The following lemma and corollary are a generalisation of \cite[Lemma 1.6(b)]{Pos} and \cite[Lemma 2.1(a)]{Pos}.

\begin{lem} \label{L:rjcontra}
Let $u:R \to U$ be a flat injective ring epimorphism with associated Gabriel topology $\G$. Then for every $J \in \G$, every $R/J$-module $M$ is a $u$-contramodule.
\end{lem}
\begin{proof}
To see that $\Hom_R(U,M)=0$, take $f:U \to M$. Then $f(U) = f(JU) = Jf(U) =0$ as $J$ annihilates $M$.

As $\Tor^R_i(R/J, U)=0$ and $R \to R/J$ is a ring epimorphism, one has that the following isomorphism.
\[
\Ext_R^1(U, M) \cong \Ext_{R/J}^1 (R/J \otimes_R U, M) =0
\]
\end{proof}
\begin{cor}\label{C:compl-contra} 
Let $u:R \to U$ be a flat injective ring epimorphism. Then $\Lambda(R)$ is a $u$-contramodule.
\end{cor}
\begin{proof} This follows immediately by Lemma~\ref{L:rjcontra} and by the closure properties of $u$-contramodules in Lemma~\ref{L:geiglenz}.
\end{proof}

\begin{lem}\label{L:coker-nu}
Let $u:R \to U$ be a flat injective ring epimorphism. Then the cokernel of $\nu:R \to \mathfrak{R}$ is a $U$-module.
\end{lem}
\begin{proof} Recall that $h_u(R)$ is the $u$-h-divisible submodule of $R$ and $\delta$ is as in sequence (\ref{eq:3contra}). Consider the following commuting diagram.
\[ \xymatrix{
0 \ar[r] &R/h_u(R) \ar[r]^\nu \ar@{=}[d] 	& \mathfrak{R} \ar[r] \ar[d]^\cong_\delta & \Coker(\nu) \ar[r] \ar[d] &0\\
0 \ar[r] &R/h_u(R) \ar[r] 		& \Ext^1_R(K,R) \ar[r]  		& \Ext^1_R(U,R) \ar[r]  &0}
\]
By the five-lemma, the last vertical arrow is an isomorphism, so $\Coker(\nu) \cong \Ext^1_R(U,R)$ which is a $U$-module, as required.

\end{proof}
\subsection{The isomorphism between the $\G$-completion of $R$ and $\End(K)$}
We now prove the main result of this section.
\begin{prop}\label{P:ringiso}
Let $u:R \to U$ be a flat injective ring epimorphism. Using the notation of Subsection  6.1 the homomorphism  $\alpha: \Lambda(R) \to \mathfrak{R}$ is a ring isomorphism.
\end{prop}
\begin{proof}
From $(\ast)$ we have the following commuting triangle:
\[ \xymatrix{
R \ar[d]_{\nu}  \ar[r]^\lambda& \Lambda(R) \ar[dl]^{\alpha} \\
\mathfrak{R} &  }
\] 
From sequences (ii) and (iii) we have the following exact sequence.
\[
0 \to h_u(R) \to R \overset{\nu} \to \mathfrak{R} \to \Coker(\nu) \to 0
\]
where $h_u(R)$ is $u$-h-divisible and  $\Coker(\nu)$ is a $U$-module by Lemma~\ref{L:coker-nu}. Both $\Lambda(R)$ and $\mathfrak{R}$ are $u$-contramodules so one can apply Lemma~\ref{L:oneten} to the two triangles below. That is, firstly, there exists a unique map $\beta$ such that $\beta \nu = \lambda$, and secondly by uniqueness, the identity on $\mathfrak{R}$ is the only homomorphism that makes the triangle on the right below commute.
\[ \xymatrix{
 R \ar[d]_{\lambda}  \ar[r]^\nu& \mathfrak{R}  \ar[dl]^{\beta} & R \ar[d]_{\nu}  \ar[r]^\nu& \mathfrak{R}  \ar[dl]^{\text{id}_\mathfrak{R}} \\
 \Lambda(R) &&  \mathfrak{R} & }
\]
It follows that since $ \alpha \beta \nu = \alpha \lambda= \nu $, by uniqueness $\alpha \beta = \text{id}_\mathfrak{R}$. Therefore, $\alpha$ is surjective. It was shown in Lemma~\ref{L:mapfacts} that $\alpha$ is injective, hence $\alpha$ is an isomorphism.

It remains to see that $\alpha$ is a ring homomorphism. First note that if $z\in U$, $s\in R$ and $Jz \subseteq R $, then also $J(sz) \subseteq R$, that is $J \subseteq \Ann_R(sz +R)$. Let $\tilde{r}= (r_J+J)_{J \in \G}$ and $\tilde{s}= (s_J+J)_{J \in \G}$ denote elements of $\Lambda(R)$. Let $L$ denote $\Ann_R(z +R)$ and $L_s$ denote $\Ann_R(sz +R)$ for a fixed $z+R$ and note that $L \subseteq L_s$.
\[
\alpha(\tilde{r} \cdot \tilde{s} ):K \to K: z+R \mapsto r_L s_L z +R
\]
\[
\alpha(\tilde{r})   \alpha(\tilde{s}) = (K \overset{\tilde{r}}\to K)  (K \overset{\tilde{s}}\to K) : z +R \mapsto s_L z +R \mapsto r_{L_s}s_Lz+R
\]
Then clearly $r_{L_s} - r_L \in L_s$, so the endomorphisms $\alpha(\tilde{r} \cdot \tilde{s} )$ and   $\alpha(\tilde{r})   \alpha(\tilde{s})$ are equal.
\end{proof}

The following lemma will be useful when passing from the ring $R$ to the complete and separated topological ring $\mathfrak{R}$.
\begin{lem} \label{L:discreteiso}
Let $u:R \to U$ be a flat injective ring epimorphism with associated Gabriel topology $\G$. The $R$-module $R/J$ is isomorphic to $\mathfrak{R}/ J\mathfrak{R}$ and to  $\Lambda(R)/J\Lambda(R)$, for every $J\in \G$.
\end{lem}
\begin{proof}
$\mathfrak{R}/ J\mathfrak{R}$ and $\Lambda(R)/J\Lambda(R)$ are isomorphic by Proposition~\ref{P:ringiso}. Both $R/J$ and $\mathfrak{R} / J \mathfrak{R}$ are $R/J$-modules, hence are $u$-contramodules by Lemma~\ref{L:rjcontra} 
and we can imply Lemma~\ref{L:oneten} to $\nu\colon R \to \mathfrak{R}$ to find that there exists a unique $f$ such that the left triangle below commutes. The map $f$ induces $\bar{f}$ since $J\mathfrak{R}\subseteq \Ker f$, so the right triangle below also commutes.
\[ \xymatrix{
R \ar[d]_{p}  \ar[r]^\nu& \mathfrak{R} \ar[dl]^{f} & \mathfrak{R} \ar[d]_{f}  \ar[r]^{\pi \hspace{15pt}}& \mathfrak{R} / J\mathfrak{R}  \ar[dl]^{\bar{f} }\\
R/J & &R/J &  }
\]
Let $\bar{\nu}$ be the map induced by $\nu$ as in the following commuting diagram. We will show that $\bar{f}$ and $\bar{\nu}$ are mutually inverse.

\[ \xymatrix{
R \ar[r]^\nu \ar[d]_{p} & \mathfrak{R} \ar[d]_{\pi}\\
R/J \ar[r]^{\bar{\nu}} & \mathfrak{R} / J\mathfrak{R}}
\]
We have that $\pi \nu = \bar{\nu} p$, and so using the above commuting triangles it follows that
$\bar{f} \bar{\nu} p = \bar{f} \pi \nu =f \nu =p $. As $p$ is surjective, $\bar{f} \bar{\nu} = \text{id}_{R/J}$. We now show that $\bar{\nu} \bar{f} = \text{id}_{{\mathfrak{R}}/J\mathfrak{R}}$.

\[ \xymatrix{
R \ar[d]_{\pi \nu}  \ar[r]^\nu&\mathfrak{R} \ar[dl]^{h} \\
\mathfrak{R} / J\mathfrak{R} & }
\]
By uniqueness, $\pi$ is the unique map that fits into the triangle above, that is $\pi \nu = h \nu$ implies that $h = \pi$. So, 
\[
\pi \nu = \bar{\nu} p = \bar{\nu} f \nu = \bar{\nu} \bar{f} \pi \nu
\]  
Therefore $\pi = \bar{\nu} \bar{f} \pi$, and as $\pi$ is surjective, $\bar{\nu} \bar{f} = \text{id}_{{\mathfrak{R}}/J\mathfrak{R}}$ as required. 
\end{proof}
\begin{prop}\label{P:topologies-2}  If $V$ is an open ideal in the topology of $\mathfrak{R}=\End_R(K)$, then there is $J\in \G$ and a surjective ring homomorphism $R/J\to \mathfrak{R}/V$.
\end{prop}
\begin{proof} By the definition of the topology on $\mathfrak{R}$, if $V$ is an open ideal, then by Proposition~\ref{P:ringiso}, $W=\alpha^{-1}(V)$ is an open ideal in the projective limit topology of $\Lambda(R)$. Hence by Remark~\ref{R:topologies}, there is $J\in \G$ such that $W\supseteq \Lambda(R)J$. By Lemma~\ref{L:discreteiso} there is a surjective ring homomorphism $R/J\to \mathfrak{R}/V.$
\end{proof} 
 \section{When a $\G$-divisible class is enveloping}\label{S:enveloping}

For this section, $R$ will always be a commutative ring. Fix a flat injective ring epimorphism $u$ and an exact sequence
\[
0 \to R \overset{u}\to U \to K \to 0.
\]
Denote  by $\G$ the Gabriel topology arising from the flat ring epimorphisms $u$. We let $\mSpec{R}$ denote the collection of maximal ideals of $R$.

The aim of this section is to show that if $\D_\G$ is enveloping then for each $J \in \G$ the ring $R/J$ is perfect. It will follow from Section~\ref{S:properfect} that also $\mathfrak{R}$ is pro-perfect. 

We begin by showing that for a local ring $R$ the rings $R/J$ are perfect, before extending the result to all commutative rings by showing that all $\G$-torsion modules (specifically the $R/J$ for $J \in \G$) are isomorphic to the direct sum of their localisations.

In Lemma~\ref{L:R-env}, it was shown that if $\varepsilon:R \to D$ is a $\D_\G$-envelope of $R$ in $\ModR$, then $D$ must be $\G$-torsion-free. Furthermore, if $\G$ arises from a perfect localisation $u:R \to U$ and $R$ has a $\D_\G$-envelope, then the following proposition allows us to work in the setting that $\D_\G = \Gen U$, thus $(\A, \D_\G)$ is the $1$-tilting cotorsion pair associated to the $1$-tilting module $U \oplus K$ (see Remark~\ref{R:pdU=1}).
 
 Combined with \cite[Proposition 5.4]{H}, the following proposition provides a generalisation of \cite[Theorem 1.1]{AHHT1}.
More precisely, the propositions show that conditions (1),(4), and (6) in \cite[Theorem 1.1]{AHHT1} hold also in our more general context. The equivalence of (1),(2), and (3) of \cite[Theorem 1.1]{AHHT1} was already shown in more generality in \cite{AS}.
 
 \begin{prop} \label{P:pd1}
Let $u:R \to U$ be a (non-trivial) flat injective ring epimorphism and suppose $R$ has a $\D_\G$-envelope. Then $\pdim_R U \leq 1$.
\end{prop}
\begin{proof}
Let 
\[
0 \to R \overset{\varepsilon}\to D \to D/R \to 0  \eqno(** )
\]
denote the $\D_\G$-envelope of $R$. First we claim that $D$ is a $U$-module by showing that $D$ is $\G$-closed, or that $D \cong U \otimes_R D$. Consider the following exact sequence.
\[
0 \to \Tor^R_1( D, K) \to D \to D \otimes_R U \to D \otimes_R K \to 0
\]
Therefore we must show that $\Tor^R_1(D, K) = 0 = D\otimes_R K$. As $D$ is $\G$-divisible and $K$ is $\G$-torsion it follows that $D \otimes_R K =0$. By Lemma~\ref{L:R-env} $D$ is $\G$-torsion-free, hence $D\cong D\otimes_R U$ and $D$ is a $U$-module.
The cotorsion pair $(\A, \D_\G)$ is complete, which implies that the $R$-module $D/R$ is in $\A$, so $\pdim_R D/R \leq 1$. From the short exact sequence {\rm $(**)$} it follows that also $\pdim_R D \leq 1$. Consider the following short exact sequence of $U$-modules 
\[
0 \to U \to D \otimes_R U \cong D \to D/R \otimes_R U \to 0
\]
We now claim that $D/R \otimes_R U$ is $U$-projective. Take any $Z \in \UMod$ and note that $Z \in \D_\G$. Then $0=\Ext^1_R(D/R, Z)\cong\Ext^1_U(D/R\otimes_R U, Z)$. Therefore the short exact sequence above splits in $\ModU$ and so $U$ is a direct summand of $D$ also as an $R$-module, and the conclusion follows. 
\end{proof}

\begin{cor}\label{C:U-envelope}
Let $u:R \to U$ be a (non-trivial) flat injective ring epimorphism and suppose $R$ has a $\D_\G$-envelope. Then 
\[
0 \to R \overset{u} \to U \to K \to 0
\]
is a $\D_\G$-envelope of $R$.
\end{cor}
\begin{proof}
By Proposition~\ref{P:pd1} $\pdim U \leq 1$, so from the discussion in Section~\ref{S:gab-top}, $U \oplus K$ is a $1$-tilting module such that $(U \oplus K)^\perp = \D_\G$. Thus $K \in \A$ and so $u$ is a $\D_\G$-preenvelope. To see that $u$ is an envelope, note that $\Hom_R(K,U)=0$, so by Lemma~\ref{L:identity-env}, if $u = f u$, then $f = \text{id}_U$ is an automorphism of $U$, thus $u$ is a $\D_\G$-envelope as required. 
\end{proof}

Later in Example~\ref{Ex:T-not-env} we give an example of a ring $R$ and $1$-tilting cotorsion class $\T$ where $R$ has a $\T$-envelope, but $\T$ is not enveloping. This result uses our characterisation of the rings over which a $1$-tilting class $\T$ is enveloping in Theorem~\ref{T:characterisation}.

 We now begin by showing that when $R$ is a commutative local ring, if $\D_\G$ is enveloping in $\ModR$ then for each $J \in \G$, $R/J$ is a perfect ring. We will use the ring isomorphism $\alpha: \Lambda(R) \cong \mathfrak{R}$ of Proposition~\ref{P:ringiso}.
\begin{lem}\label{L:Kindecomp}
Let $R$ be a commutative local ring and $u:R \to U$ a flat injective ring epimorphism and let $K$ denote $U/R$. Then $K$ is indecomposable.
\end{lem}
\begin{proof}
It is enough to show that every idempotent of $\End_R(K)$ is either the zero homomorphism or the identity on $K$. Let $\m$ denote the maximal ideal of $R$. Take a non-zero idempotent $e \in \End_R(K)$. Then there is an associated element $\alpha^{-1}(e)=\tilde{r}:=(r_J +J)_{J \in \G} \in \Lambda(R)$ via the ring isomorphism $\alpha:\Lambda(R) \cong \mathfrak{R}$ of Proposition~\ref{P:ringiso}. Clearly $\tilde{r}$ is also non-zero and an idempotent in $\Lambda(R)$. We will show this element is the identity in $\Lambda(R)$.

As $\tilde{r}$ is non-zero, there exists a $J_0 \in \G$ such that $r_{J_0} \notin J_0$. Also, $\tilde{r} \cdot \tilde{r} - \tilde{r} =0$, hence
\[
r_{J_0}r_{J_0} - r_{J_0} = r_{J_0}(r_{J_0} -1_R) \in J_0.
\] 
We claim that $r_{J_0}$ is a unit in $R$. Suppose not, then $r_{J_0} \in \m$, hence $r_{J_0} -1_R$ is a unit, which implies that $r_{J_0} \in J_0$, a contradiction.

Consider some other $J \in \G$ such that $J \neq R$. $r_{J \cap J_0} - r_{J_0} \in J_0$, hence $r_{J \cap J_0} \notin J_0$. Therefore, by a similar argument as above, $r_{J \cap J_0}$ is a unit in $R$. As $r_{J \cap J_0} - r_{J} \in J$ and $r_{J \cap J_0}$ is a unit, $r_J \notin J$. Therefore by a similar argument as above $r_J$ is a unit in $R$ for each $J \in \G$ and we conclude that $\tilde{r}$ is a unit in $\Lambda(R)$. 

Finally, as $r_J(r_J - 1_R) \in J$ for every $J$, and $\tilde{r}:=(r_J +J)_{J \in \G} $ is a unit, it follows that $r_J - 1_R \in J$ for each $J$, implying that $\tilde{r}$ is the identity in $\Lambda(R)$. 
\end{proof}
\begin{prop}\label{P:localperfect}
Let $R$ be a commutative local ring and consider the $1$-tilting cotorsion pair $(\A, \D_\G)$ induced by the flat injective ring epimorphism $u:R \to U$. If $ \D_\G$ is enveloping in $\RMod$, then $R/J$ is a perfect ring for every $J \in \G$. 
\end{prop} 
\begin{proof}
Let $\m$ denote the maximal ideal of $R$. As $R$ is local, to show that $R/J$ is perfect it is enough to show that for every sequence of elements $\{a_1, a_2, \dots, a_i, \dots \}$ with $a_i \in \m \setminus J$, there exists an $m >0$ such that the product $a_1 a_2 \cdots a_m \in J$ (that is  $\m/J$ is T-nilpotent) by Proposition~\ref{P:perfect}.

Fix a $J\in \G$ and take $\{a_1, a_2, \dots, a_i, \dots \}$ as above. Consider the following preenvelope of $R/a_iR$.
\[
0 \to R/a_iR \hookrightarrow U/a_iR \to K \to 0
\]
As $R$ is local, by Lemma~\ref{L:Kindecomp}, $K$ is indecomposable, and as $R/a_iR$ is not $\G$-divisible this is an envelope of $R/a_iR$.

We will use the T-nilpotency of direct sums of envelopes from Theorem~\ref{T:Xu-sums}.
Consider the following countable direct sum of envelopes of $R/a_iR$ which is itself an envelope by Theorem~\ref{T:Xu-sums}~(i). 
 \[ 0 \to \bigoplus_{\substack{
   i>0
  }} R/a_iR \hookrightarrow \bigoplus_{\substack{
   i>0
  }}U/a_iR \to \bigoplus_{\substack{
   i>0
  }}K \to 0 \] 
For each $i>0$, we define a homomorphism $f_i:U/a_iR \to U/a_{i+1}R$ between the direct summands to be the multiplication by the element $a_{i+1}$. 
%
Then clearly $R/a_iR \subseteq U/a_iR$ vanishes under the action of $f_i = \dot{a}_{i+1}$, hence we can apply Theorem~\ref{T:Xu-sums}~(ii) to the homomorphisms $\{f_i\}_{i>0}$. So, for every $z + a_1R \in U/a_1R$, there exists an $n>0$ such that 
\[ f_n   \cdots   f_2   f_1 (z+a_1R) = 0 \in U/a_{n+1}R,\]
which can be rewritten as
 \[
 a_{n+1} \cdots a_3 a_2 (z) \in a_{n+1}R.
 \]
By Lemma~\ref{L:finmanyann}, there exist $z_1, z_2, \dots , z_n \in U$ such that 
\[
\bigcap_{\substack{
   0 \leq j \leq n}}
   \Ann_R(z_j +R) \subseteq J.
\]
Let $\Omega = \{z_1, z_2, \dots , z_n\}$. For each $z_j$, there exists an $n_j$ such that $a_{n_j+1} \cdots a_3 a_2$ annihilates $z_j$. That is,
\[a_{n_j+1} \cdots a_3 a_2 (z_j) \in a_{n_j+1}R\subseteq R.\]
We now choose an integer $m$ such that $a_{m} \cdots a_3 a_2$ annihilates all the $z_j$ for $a \leq j \leq n$. Set $m = max\{n_j\mid j=1,2\dots, n\}$.
Then this $m$ satisfies the following, which finishes the proof.
\[
a_m a_{m-1} \cdots a_3 a_2 \in \bigcap_{\substack{
   0 \leq j \leq n}}
   \Ann_R(z_j +R) \subseteq J
\]
\end{proof}
Now we extend the result to general commutative rings. Our assumption is that the Gabriel topology $\G$ is arises from a perfect localisation $u:R \to U$ and that the associated $1$-tilting class $\D_\G$ is enveloping in $\RMod$.

\begin{nota} \label{N:simple-env}
There is a preenvelope of the following form induced by the map $u$.
\[
0 \to R/\m \to U / \m \to K \to 0
\]
 Let the following sequence denote an envelope of $R/\m$.
\[
0 \to R/\m \to D(\m) \to X(\m) \to 0
\]
By Proposition~\ref{P:Xu-env}, $D(\m)$ and $X(\m)$ are direct summands of $U/\m$ and $K =U/R$ respectively. For convenience we will consider $R/\m$ as a submodule of $D(\m)$ and $X(\m)$ as a submodule of $K$.
\end{nota}
\begin{rem}\label{R:tors-facts}\emph{
\begin{enumerate}
\item[(i)] Note that for every maximal ideal $\m$ of $R$,  $R/\m$ is $\G$-divisible if and only if, for every $J \in \G$, $J+\m=R$ if and only if for every $J \in \G$, $J\nsubseteq\m$ if and only if $\m\notin\G$. Therefore, we will only consider the envelopes of $R/\m$ where $\m \in \G$. The modules $D(\m)$ and $X(\m)$ will always refer to the components of the envelope of some $R/\m$ where $\m \in \G$.
Additionally, as $R/\m$ is also an $R_\m$-module, it follows by Proposition~\ref{P:B-envelopes} that $D(\m)$ and $X(\m)$ are also $R_\m$-modules.
\item[(ii)] For every $J \in \G$, $(R/J)_\m =0$ if and only if $J \nsubseteq \m$.
\item[(iii)] If $M$ is a $\G$-torsion $R$-module, then $M_\m=0$ for every $\m \notin \G$ which follows by (ii).
\end{enumerate}}
\end{rem}
The following lemma allows us to use Proposition~\ref{P:localperfect} to show that if $\D_\G$ is enveloping in $R$, all localisations $R_\m/J_\m$ are perfect rings where $\m$ is a maximal ideal in $\G$ and $J \in \G$. 

If $R$ is a commutative ring with a maximal ideal $\m$ and $\C$ a class of $R$-modules, we define $\C_\m$ to be the class consisting of localisations of modules in $\C$. That is, $\C_\m = \{C_\m \mid C \in \C\}$.

\begin{lem} \label{L:localfacts}
Let $R$ be a commutative ring and consider the $1$-tilting cotorsion pair $(\A, \D_\G)$ induced from the flat injective ring epimorphism $u:R \to U$. Fix a maximal ideal $\m$ of $R$ and let $u_\m: R_\m \to U_\m$ be the corresponding flat injective ring epimorphism in $\ModR_\m$. Then the following hold.
\begin{enumerate}
\item[(i)] $K_\m = 0$ if and only if $\m \notin \G$.
\item[(ii)] The induced Gabriel topology of $u_\m$ denoted \[\G(\m) = \{L \leq R_\m : L U_\m = U_\m \}\] contains the localisations $\G_\m = \{J_\m : J \in \G \}$.
\item[(iii)]  Suppose $\pdim U \leq 1$. Then $(\D_\G)_\m$ is the $1$-tilting class associated to the flat injective ring epimorphism $u_\m: R_\m \to U_\m$. That is, $(\D_{\G})_\m = \D_{\G(\m)}$.
\item[(iv)] If $\D_\G$ is enveloping in $\ModR$, then $\D_{\G(\m)}$ is enveloping in $\ModR_\m$.
\end{enumerate}
\end{lem}
\begin{proof}
\begin{enumerate}
\item[(i)] Since $K$ is $\G$-torsion, $\m \notin \G$ implies $K_\m=0$ by Remark~\ref{R:tors-facts}~(iii). For the converse, suppose $K_\m=0$. If $\m \in \G$ then $R_\m \cong U_\m = \m_\m U_\m \cong \m_\m R_\m$, a contradiction.
Note that if $\m \notin \G$ the rest of the lemma follows trivially. 
\item[(ii)] Take $J_\m \in \G_\m$. Then $ R_\m/ J_\m \otimes_R U_\m \cong (R/J \otimes_R U) \otimes_R R_\m = 0 $, so $J_\m \in \G(\m)$.
\item[(iii)] If $\pdim U \leq 1$, $U \oplus K$ is a $1$-tilting module for $\D_\G$ by \cite[Theorem 5.4]{H}. That $(\D_\G)_\m$ is the $1$-tilting class associated to the $1$-tilting module $(U \oplus K)_\m$ is \cite[Proposition 13.50]{GT12}, therefore $\Gen (U_\m) = (\D_\G)_\m$ in $\ModR_\m$. As $u_\m :R_\m \to U_\m$ is a flat injective ring epimorphism and $\pdim_{R_\m}U_\m \leq 1$ the $1$-tilting classes $\Gen(U_\m)$ and $\D_{\G(\m)} $ coincide in $\ModR_\m$ again by \cite[Theorem 5.4]{H}. Thus $(\D_{\G})_\m = \D_{\G(\m)}$.
\item[(iv)] Assume that $\D_\G$ is enveloping in $\ModR$ and take some $M \in \ModR_\m$ with the following $\D_\G$-envelope.
\[
0 \to M \to D \to X \to 0
\]
We claim that $M$ has a $\D_{\G(\m)}$-envelope in $\ModR_\m$. Since $M \in \ModR_\m$, $D$ and $X$ are $R_\m$-modules by Proposition~\ref{P:B-envelopes}. By Proposition~\ref{P:pd1} $\pdim U \leq 1$. By (iii), $(\D_{\G})_\m = \D_{\G(\m)}$ so $D \in \D_{\G(\m)}$. Moreover, $X \in {}^\perp \D_{\G(\m)}$ as $X \in \Add(U \oplus K) \cap \ModR_\m$ so $X \in \Add(U_\m \oplus K_\m)$.  Since $R \to R_\m$ is a ring epimorphism, any direct summand of $D$ which contains $M$ in $\ModR_\m$ would also be a direct summand in $\ModR$. Thus we conclude that $0\to M\to D\to X\to 0$ is a $\D_{\G(\m)}$-envelope of $M$ in $\ModR_\m$.
\end{enumerate}
\end{proof}
By the above lemma, if $\D_\G$ is enveloping in $\ModR$, then $\D_{\G(\m)}$ is enveloping in $\ModR_\m$. Next we show that, under our enveloping assumption, all $\G$-torsion modules are isomorphic to the direct sums of their localisations at maximal ideals. 

The proof of the following lemma uses an almost identical argument to the proof of Lemma~\ref{L:torsion-env}.

\begin{lem}\label{L:Xmnilpotent} 
Let $u:R \to U$ be a flat injective ring epimorphism, $\G$ the associated Gabriel topology and suppose that $\D_\G$ is enveloping. Let $D(\m)$ and $X(\m)$ be as in Notation~\ref{N:simple-env} and fix a maximal ideal $\m \in \G$. For every element $d \in D(\m)$ and every element $a \in \m$, there is a natural number $n > 0$ such that $a^n d = 0$. Moreover, for every element $x \in X(\m)$ and every element $a \in \m$, there is a natural number $n > 0$ such that $a^n x = 0$.
\end{lem}
\begin{proof}
We will use the T-nilpotency of direct sums of envelopes as in Theorem~\ref{T:Xu-sums}~(ii).
Consider the following countable direct sum of envelopes of $R/\m$ which is itself an envelope by Theorem~\ref{T:Xu-sums}~(i). 
 \[ 0 \to \bigoplus_{\substack{
   i>0
  }} (R/\m)_{i} \to \bigoplus_{\substack{
   i>0
  }}D(\m)_{i} \to \bigoplus_{\substack{
   i>0
  }}X(\m)_{i} \to 0 \] 
For a fixed element $a \in \m$, we choose the homomorphisms $f_i:D(\m)_{i} \to D(\m)_{i+1}$ between the direct summands to be multiplication by $a$.   
Then clearly $R/\m \subseteq D(\m)$ vanishes under the action of $f_i = \dot{a}$, hence we can apply Xu's Theorem: for every $d \in D(\m)$, there exists an $n$ such that 
\[ f_n   \cdots   f_2   f_1 (d) = 0 \in D(\m)_{n+1}.\]
 Since each $f_i$ acts as multiplication by $a$, for every $d \in D$ there is an integer $n$ for which $a^n d = 0$, as required.
 
It is straightforward to see that $X(\m)$ has the same property as $X(\m)$ is an epimorphic image of $D(\m)$. 
 \end{proof}
\begin{lem}\label{L:suppXm} 
Let $u:R \to U$ be a flat injective ring epimorphism and suppose $\D_\G$ is enveloping. Let $\m \in \G$ and let $X(\m)$ be as in Notation~\ref{N:simple-env}. The support of $X(\m)$ is exactly $\{ \m \}$, and each $X(\m) \cong X(\m)_\m$ is  $K_\m$.
\end{lem}
\begin{proof}
We claim that $X(\m)$ is non-zero. Otherwise, $X(\m)=0$ would imply that $R/\m$ is $\G$-divisible, so $R/\m = \m(R/\m) = 0$, a contradiction.

Consider a maximal ideal $\n \neq \m$. Take an element $a \in \m \setminus \n$. Then for any $x \in X(\m)$, $a^n x = 0$ for some $n > 0$, by Lemma~\ref{L:Xmnilpotent} and since $a$ is an invertible element in $R_\n$, $x$ is zero in the localisation with respect to $\n$. This holds for any element $x \in X(\m)$, hence $X(\m)_\n = 0$. 
 
It follows that since $X(\m)$ is non-zero, $X(\m)_\m \neq 0$. As mentioned in Remark~\ref{R:tors-facts}, $X(\m)$ is an $R_\m$-module and since $X(\m)$ is a direct summand of $K$, $X(\m)$ is a direct summand of $K_\m$ which is indecomposable, by Lemma~\ref{L:Kindecomp}. Therefore $X(\m)$ is non-zero and is isomorphic to $K_\m$.
\end{proof}

\begin{lem}\label{L:directsumxm}
Let $u:R \to U$ be a flat injective ring epimorphism and suppose $\D_\G$ is enveloping. Then the sum of the submodules $X(\m)$ in $K$ is a direct sum.
\[
\sum_{\substack{
   \m \in \G
  }} X(\m)
  = 
  \bigoplus_{\substack{
   \m \in \G
  }} X(\m) 
\]
\end{lem} 
\begin{proof}
Recall that $X(\m)$ is non-zero only for $\m \in \G$ by Remark~\ref{R:tors-facts}. Consider an element
\[
x \in X(\m) \cap  \sum_{\substack{
  \n \neq \m\\ \n \in \G
  }} X(\n).
\]
We will show that this element must be zero. By Lemma~\ref{L:suppXm}, since $x \in X(\m)$, $x$ is zero in the localisation with respect to all maximal ideals $\n \neq \m$. But $x$ can also be written as a finite sum of elements $x_i \in X(\n_i)$, each of which is zero in the localisation with respect to $\m$, by Lemma~\ref{L:suppXm}.  Therefore, $(x)_\n =0$ for all maximal ideals $\n$, hence $x = 0$ .
\end{proof}
\begin{prop}\label{P:directsumk}
Let $u:R \to U$ be a flat injective ring epimorphism and suppose $\D_\G$ is enveloping. The module $K$ can be written as a direct sum of its localisations $K_\m$, as follows.
\[
 K \cong  \bigoplus_{\substack{
   \m \in \G
  }} K_\m =  \bigoplus_{\substack{
   \m \in \mSpec{R}
  }} K_\m
\]
\end{prop} 
\begin{proof}
From Lemma~\ref{L:directsumxm}, we have the following inclusion.
\[
  \bigoplus_{\substack{
   \m \in \G
  }} X(\m) \leq K
\]
To see that this is an equality we show that these two modules have the same localisation with respect to every $\m$ maximal in $R$. Recall that by Lemma~\ref{L:localfacts}(i) if $\n$ is maximal, then $K_\n = 0$ if and only if $\n \notin \G$ and by Lemma~\ref{L:suppXm},  $\text{Supp}(X(\m)) = \{ \m\}$. Using these facts, it follows that for $\n \notin \G$,  $K_\n = 0 = ( \bigoplus_{\substack{
   \m \in \G
  }} X(\m))_\n$. Similarly, if $\m \in \G$, then $K_\m = X(\m)_\m$. Hence,
  \[
  \bigoplus_{\substack{
   \m \in \G
  }} X(\m) = K.
\]
Since $K_\m = X(\m)_\m$, it only remains to see that $X(\m) \cong X(\m)_\m$, which follows from Remark~\ref{R:tors-facts}.
\end{proof}

\begin{cor}\label{C:torsion-decomposes}
Let $u:R \to U$ be a flat injective ring epimorphism and suppose $\D_\G$ is enveloping. Then for every $\G$-torsion module $M$, the following isomorphism holds.
 \[
M \cong \bigoplus_{\substack{
   \m \in \G
  }} M_\m = \bigoplus_{\substack{\m \in \mSpec R}} M_\m
\]
Furthermore, it follows that for every $J \in \G$, $J$ is contained in only finitely many maximal ideals of $R$. 
\end{cor} 
\begin{proof}
For the first isomorphism, recall that if an $R$-module $M$ is $\G$-torsion, then $M \cong \Tor^R_1(M, K)$. Also, note that in this case, $M_\m \cong \Tor^R_1(M, K)_\m \cong  \Tor^{R_\m}_1(M_\m, K_\m) \cong \Tor^{R}_1(M, K_\m)$. Hence we have the following isomorphisms.
\[
M \cong \Tor^R_1(M, K) \cong \Tor^R_1(M, \bigoplus_{\substack{
   \m \in \G
  }} K_\m) \cong \bigoplus_{\substack{\m \in \G}} \Tor^R_1(M, K_\m) \cong \bigoplus_{\substack{\m \in \G}} M_\m
\]
The fact that
\[
\bigoplus_{\substack{
\m \in \G
}} M_\m = \bigoplus_{\substack{\m \in \mSpec R}} M_\m
\]
follows from Remark~\ref{R:tors-facts}~(iii).
\\
For the final statement of the proposition, one only has to replace $M$ with the $\G$-torsion module $R/J$ where $J \in \G$. Hence as $R/J$ is cyclic, it cannot be isomorphic to an infinite direct sum. 
Therefore, $(R/J)_\m$ is non-zero only for finitely many maximal ideals and the conclusion follows. 
\end{proof} 
We are now in the position to show the main results of this section.
 \begin{thm} \label{T:rjperfect}
 Let $u:R \to U$ be a flat injective ring epimorphism and suppose $\D_\G$ is enveloping. Then $R/J$ is a perfect ring for every $J \in \G$.
\end{thm}
\begin{proof}
By Corollary~\ref{C:torsion-decomposes}, every $R/J$ is a finite product of local rings $R_\m/J_\m$. Additionally as $(\D_\G)_\m$ is enveloping in $\ModR_\m$ by Lemma~\ref{L:localfacts} each $R_\m/J_\m$ is a perfect ring by Proposition~\ref{P:localperfect}. Therefore, by Proposition~\ref{P:perfect}, $R/J$ itself is perfect.
\end{proof} 
 \begin{thm}\label{T:EndK-properfect}
Let $u:R \to U$ be a flat injective ring epimorphism and suppose $ \D_\G$ is enveloping in $\ModR$. Then the topological ring $\mathfrak{R}=\End(K)$ is pro-perfect.
 \end{thm}
 \begin{proof}
 Recall that the topology of $\mathfrak{R}$ is given by the annihilators of finitely generated submodules of $K$, so that $\mathfrak{R}=\End_R(K)$ is separated and complete in its topology.
 Let  $V$ be an open ideal in the topology of $\mathfrak{R}$. By Proposition~\ref{P:topologies-2} there is $J\in \G$ and a surjective ring homomorphism $R/J\to \mathfrak{R}/V$. By Theorem~\ref{T:rjperfect} $R/J$ is a perfect ring and thus so  are the quotient rings $\mathfrak{R}/V$. \end{proof}

\section{$\D_\G$ is enveloping if and only if $\mathfrak{R}$ is pro-perfect}\label{S:properfect}
Suppose that $u:R \to U$ is a commutative flat injective ring epimorphism where $\pdim_R U \leq 1$ and denote $K = U/R$. In this section we show that if the endomorphism ring $\mathfrak{R} = \End_R(K)$ is pro-perfect, then $\D_\G$ is enveloping in $\ModR$. So combining with the results in the Section~\ref{S:enveloping} we obtain that $\D_\G$ is enveloping if and only if $\pdim U \leq 1$ and $\mathfrak{R}$ is pro-perfect.

Recall that if $\pdim U \leq 1$, $(\A, \D_\G)$ denotes the $1$-tilting cotorsion pair associated to the $1$-tilting module $U \oplus K$.
The following theorem of Positselski is vital for this section. 
\begin{thm}\label{T:pos-addK}(\cite[Theorem 13.3]{BP4})
Suppose $R$ is a commutative ring and $u:R \to U$ a flat injective ring epimorphism with $\pdim_R U \leq 1$. Then the topological ring $\mathfrak{R}=\End(K)$ is pro-perfect if and only if $\varinjlim \Add(K) = \Add(K)$. 
\end{thm}
A second crucial result that we will use is the following.
\begin{thm}\label{T:Enochs2} (\cite[Theorem 2.2.6]{Xu}) Assume that $\C$ is a class of modules closed under direct limits and extensions. If a module
$M$ admits a special $\C^{\perp_1}$-preenvelope with cokernel in $\C$, then $M$ admits a $\C^{\perp_1}$-envelope.
\end{thm}
 We now show that if $\mathfrak{R}$ is pro-perfect, then $\Add (K)$ does in fact satisfy the conditions of Theorem~\ref{T:Enochs2}. From Theorem~\ref{T:pos-addK} $\Add(K)$ is closed under direct limits. 
  Moreover, $\Add(K)$ is  closed under extensions as any short exact sequence $0 \to L \to M \to N \to 0$ with $L, N \in \Add(K)$ splits.

As the cotorsion pair $(\A, \D_\G)$ is complete, every $R$-module $M$ has an injective $\D_\G$-preenvelope, and as $\D_\G = K^\perp = (\Add (K))^\perp$, $M$ has a $(\Add (K))^\perp$-preenvelope. It remains to be seen that every $M$ has a special preenvelope $\nu$ such that $\Coker \nu \in \Add (K)$, which we will now show.

\begin{lem}\label{L:coker-in-AddK}
Suppose $u:R \to U$ is a flat injective ring epimorphism where $\pdim_R U \leq 1$. Let $(\A, \D_\G)$ be the $1$-tilting cotorsion pair associated to the $1$-tilting module $U \oplus K$. Then every module has a special $\D_\G$-preenvelope $\nu$ such that $\Coker \nu \in \Add (K)$.
\end{lem}
\begin{proof}
Take an $R$-module $M$ and consider the canonical surjection $R^{(\alpha)} \overset{p}\to M \to 0$. For every cardinal $\alpha$ the short exact sequence $0 \to R^{(\alpha)} \to U^{(\alpha)} \to K^{(\alpha)} \to 0$ is a $\D_\G$-preenvelope and is of the desired form. Consider the following pushout $Z$ of $M \gets R^{(\alpha)} \to U^{(\alpha)}$. 
\[\xymatrix{
&0 \ar[d]& 0  \ar[d]& & \\
&\ker p \ar[d] \ar@{=}[r]& \ker p \ar[d] & & \\
0 \ar[r]&R^{(\alpha)} \ar[r] \ar[d]^p& U^{(\alpha)} \ar[r] \ar[d]& K^{(\alpha)} \ar@{=}[d] \ar[r] & 0\\
0 \ar[r]&M\ar[r] \ar[d] & Z \ar[r] \ar[d]& K^{(\alpha)} \ar[r] & 0\\
&0&0 &&}
\]
The module $Z$ is in $\Gen U=\D_\G$, and so the bottom row of the above diagram is a $\D_\G$-preenvelope of $M$ of the desired form. \\
\end{proof}
The following theorem follows easily from the above discussion.
\begin{thm} \label{T:divenv}
Suppose $u:R \to U$ is a flat injective ring epimorphism with $\pdim_R U \leq 1$. If the topological ring $\mathfrak{R}$ is pro-perfect, then $\D_\G$ is enveloping in $\RMod$.
\end{thm}
\begin{proof}
 From Theorem~\ref{T:pos-addK} and Lemma~\ref{L:coker-in-AddK}, $\Add (K)$ does satisfy the conditions of Theorem~\ref{T:Enochs2}. Thus the conclusion follows, since $\D_\G=\Add(K)^\perp$.\end{proof}
 
Finally combining the above theorem with the results in Section~\ref{S:tilting-enveloping} and Section~\ref{S:enveloping} we obtain the two main results of this paper.

\begin{thm}\label{T:characterisation}
Suppose $u:R \to U$ is a commutative flat injective ring epimorphism, $\G$ the associated Gabriel topology and $\mathfrak R$ the topological ring $\End_R(K)$. The following are equivalent.
\begin{enumerate}
\item[(i)] $\D_\G$ is enveloping.
\item[(ii)] $R/J$ is a perfect ring for every $J\in \G$.
\item[(iii)] $\mathfrak{R}$ is pro-perfect.
\end{enumerate} 
It follows that $\pdim U \leq 1$. 
If $\D_\G$ is enveloping then the class $\Add (K)$ is closed under direct limits.
\end{thm}
\begin{proof}
(i)$\Rightarrow$(ii) Follows by Proposition~\ref{P:pd1} and Theorem~\ref{T:rjperfect}.

(ii)$\Rightarrow$(iii) Follows from Lemma~\ref{L:discreteiso} and Proposition~\ref{P:topologies-2}.

(iii)$\Rightarrow$(i) Follows from Theorem~\ref{T:divenv}.

That (i) implies $\pdim U \leq 1$ is Proposition~\ref{P:pd1}.

That (ii) implies $\pdim U \leq 1$ is a result of Positselski via private communication. The proof is a generalisation of \cite[Theorem 6.13]{BP1}.
\end{proof}

\begin{rem}\label{R:pos-ref}
In the original version of this paper, the assumption that $\pdim R_\G\leq 1$ was included on the right hand side of the equivalence in Theorem~\ref{T:tilting-envelope}. This has been removed as it was pointed out by Leonid Positselski (via private correspondence) that if the $R/J$ are perfect rings for $J \in \G$ and $\G$ is a perfect Gabriel topology, it follows that $\pdim R_\G \leq 1$. His proof is a generalisation of \cite[Theorem 6.13]{BP1}.
\end{rem}

 \begin{thm}\label{T:tilting-envelope} Assume that $T$ is a $1$-tilting module over a commutative ring $R$ such that the class $T^\perp$ is enveloping, and let $\G$ be the associated Gabriel topology of $\T$. Then we have the following equivalence.
  \begin{equation*}
  \T \text{ is enveloping} \Leftrightarrow \begin{cases}
    R/J \text{ is a perfect ring for each } J \in \G \\
    \G \text{ is a perfect Gabriel topology}
  \end{cases}
\end{equation*}
That is, there is a flat injective ring epimorphism $u\colon R\to U$ such that $\pdim U \leq 1$ and $U\oplus U/R$ is equivalent to $T$.
 \end{thm}
\begin{proof} ($\Rightarrow$) By Proposition~\ref{P:tilting-env}, the Gabriel topology $\G$ associated to $T^\perp$ arises from a perfect localisation. Moreover, $\psi\colon R \to R_\G$ is injective so by setting $U=R_\G$ we can apply Theorem~\ref{T:characterisation} to conclude.\\
($\Leftarrow$)One applies Theorem~\ref{T:characterisation} to conclude that $\T$ is enveloping.
\\
The last statement follows by Remark~\ref{R:pos-ref}. 
\end{proof}

The following is an application of Theorem~\ref{T:tilting-envelope}, which allows us to characterise all the $1$-tilting cotorsion pairs over a commutative semihereditary ring (for example, for the category of abelian groups). 

\begin{expl}
Let $R$ be a semihereditary ring and $(\A, \T)$ a $1$-tilting cotorsion pair in $\ModR$ with associated Gabriel topology $\G$. Then by \cite[Theorem 5.2]{H}, $\G$ is a perfect Gabriel topology. Moreover, $R/J$ is a coherent ring for $J \in \G$, so $R/J$ is a perfect ring if and only if it is artinian \cite[Theorem 3.3 and 3.4]{Ch}. As $R/J$ is artinian, it has only finitely many (finitely generated) maximal ideals and the Jacobson radical of $R/J$ is a nilpotent ideal. Therefore in this case, $\G$ has a subbasis of ideals of the form $\{\m^{k} \mid \m \in \mSpec R \cap \G, k \in \bbN \}$ and moreover all the maximal ideals of $R$ contained in $\G$ are finitely generated.


In particular, in the case of $R = \bbZ$, every $1$-tilting class $\T$ is enveloping as $\bbZ$ is semihereditary and for any proper ideal $a\bbZ$ of $\bbZ$, $\bbZ/a\bbZ$ is artinian. 

\end{expl}
The following is an example of a ring $R$ and $1$-tilting class $\T$ such that $R$ has a $\T$-envelope but $\T$ is not enveloping. 
\begin{expl}\label{Ex:T-not-env}
Let $R$ be a valuation domain with valuation $v$ and valuation group $\Gamma(R)= \bbR$, and an idempotent maximal ideal $\m= <r_n \in R \mid v(r_n)= 1/n, n \in \bbZ^{>0}>$ (see \cite[Section II.3]{FS} for details on valuation rings). Then as $Q$ is generated by $a_n^{-1}$ with $v(a_n) = n$, it follows that the field of quotients $Q$ of $R$ is countably generated and therefore of projective dimension at most one. Thus $Q \oplus Q/R$ is a $1$-tilting module and the associated Gabriel topology is made up of the principal ideals generated by the non-zero elements of $R$. Moreover, the following is a $\T$-envelope of $R$.
\[
0 \to R \to Q \to Q/R \to 0
\]
However, we claim that $\T$ is not enveloping. If $\T$ is enveloping, then by Theorem~\ref{T:characterisation} $R/sR$ is a perfect ring for each regular element $s$ in $R$. By \cite[Theorem 4.4 and Proposition 4.5]{BSa}, $R$ must be a discrete valuation domain. However, by assumption $R$ is not noetherian as $\m$ is countably generated, a contradiction. 
\end{expl}

\sectionmark{For non-injective flat ring epimorphisms}
\section{The case of a non-injective flat ring epimorphism}\label{S:notmono}
\sectionmark{For non-injective flat ring epimorphisms}
Now we extend the results of the previous section to the case of a non-injective flat ring epimorphism $u\colon R \to U$ with $K=\Coker u$.

As before, the Gabriel topology $\G_u=\{J\leq R\mid JU=U\}$ associated to $u$ is finitely generated and the class
\[\D_{\G_u}=\{M\in \ModR \mid JM=M \text{ for every } J\in\G_u\}
\]
  of $\G_u$-divisible modules is a torsion class. Moreover, by \cite{AHHr} it is a silting class, that is there is a silting module $T$ such that $\Gen(T)=\D_{\G_u}$.

The ideal $I$ will denote the kernel of $u$ and $\overline R$ the ring $R/I$ so that there is a flat injective ring epimorphism $\overline u\colon \overline R\to U$. 
  
     To $\overline{u}$, one can associate the Gabriel topology $\G_{\overline u}=\{L/I\leq \overline R\mid LU=U, I \subseteq L\}$ on $\overline{R}$ and the following class of $\overline{R}$-modules.
      \[\D_{\G_{\overline u}}=\{M\in \Mod{\overline R} \mid (L/I)M=M,  \text{ for every } L/I\in\G_{\overline u}\} \]
That is, we have that if $J \in \G_u$, then $J+I/I \in \G_{\overline u}$, and conversely if $L/I \in \G_{\overline u}$, $L \in \G_u$.

We first note the following.
\begin{lem}\label{L: I-annih-div} Every module in $\D_{\G_u}$ is annihilated by $I$, thus $\D_{\G_u}=\D_{\G_{\overline u}}$.
\end{lem}
\begin{proof} Note that $\Ker u=I$ is the $\G_u$-torsion submodule of $R$. Hence for every $b\in I$ there is $J\in \G_u$ such that $bJ=0$.
 Let $M\in \D_{\G_u}$, then $bM =bJM=0$, thus $IM=0$. We conclude that $\D_{\G_u}$ can be considered a class in $\Mod{\overline R}$ and coincides with $\D_{\G_{\overline u}}$.
 \end{proof}
 
 \begin{prop}\label{P:same-env} The class  $\D_{\G_u}$ is enveloping in $\ModR$ if and only if $\D_{\G_{\overline u}}$ is enveloping in $\Mod{\overline R}$.
 \end{prop}
 \begin{proof} Assume that $\D_{\G_u}$ is enveloping in $\ModR$ and let $\overline M\in \Mod{\overline R}$. Consider a $\D_{\G_u}$-envelope 
 $\overline{\mu}\colon \overline M\to D$ in $\ModR$. Since $R\to R/I$ is a ring epimorphism and $D$ is annihilated by $I$ by Lemma~\ref{L: I-annih-div},  it is immediate to conclude that  $\overline{\mu}$ is also a $\D_{\G_{\overline u}}$-envelope of $\overline M$.
 
  Conversely, assume that $\D_{\G_{\overline u}}$ is enveloping in $\Mod{\overline R}$. Take $M\in\ModR$ and let $\overline{\mu}\colon M/IM\to D$ be a $\D_{\G_{\overline u}}$-envelope of $ M/IM$ in $\Mod{\overline R}$.  Let $\pi\colon M\to M/IM$ be the canonical projection. We claim that $\mu=\overline{\mu} \pi$ is a $\D_{\G_u}$-envelope of $M$ in $\ModR$. Indeed, if $f\colon D\to D$ satisfies $f\mu=\mu$, then $f\overline{\mu}\pi =\overline{\mu} \pi$. As $\pi$ is a surjection, $f \overline{\mu} = \overline{\mu}$ and so $f$ is an automorphism of $D$.
  \end{proof}
  
  Note that $\End_R(K)$ coincides with $\End_{\overline R}(K)$ both as a ring and as a topological ring. It will be still denoted by $\mathfrak R$.
   Thus if $\D_{\G_u}$ is enveloping in $\ModR$ we can apply the results of the previous sections to the ring $\overline R$, in particular Theorem~\ref{T:characterisation}.
   
 \begin{thm}\label{T:non-mono} Let $u\colon R\to U$ be a commutative flat ring epimorphism with kernel $I$. Let $\G_u$ be the associated Gabriel topology and $\mathfrak R$ the topological ring $\End_R(K)$. The following are equivalent.
 \begin{enumerate}
\item[(i)] $\D_{\G_u}$ is enveloping.
\item[(ii)]$R/L$ is a perfect ring for every $L\in \G$ such that $L\supseteq I$.
\item[(iii)] $\mathfrak{R}$ is pro-perfect.
\end{enumerate} 
In particular,  $\pdim_{\overline{R}} U \leq 1$ and $U\oplus K$ is a $1$-tilting module over the ring $\overline R$ and since $\Gen(U)$ is contained in $\Mod{\overline R}$, $\D_{\G_u}=\Gen(U)$.
\end{thm}

As already noted, results from \cite{AHHr} imply that $\Gen(U)$ is a silting class in $\ModR$. Since we have that  $U\oplus K$ is a $1$-tilting module in $\Mod{\overline R}$ inducing the silting class $\Gen(U)$, it is natural to ask the following question.
\begin{ques} Is $U\oplus K$ a silting module in $\ModR$?
\end{ques}

\bibliographystyle{alpha}
\bibliography{references}
\end{document}